\documentclass[twoside]{article}
\usepackage[cp1251]{inputenc}
\usepackage[T2A]{fontenc}
\usepackage{array}
\usepackage{amssymb}
\usepackage{amsmath}
\usepackage{amsthm}
\usepackage{latexsym}
\usepackage{indentfirst}
\usepackage{bm}
\usepackage{enumerate}
\usepackage[dvips]{graphicx}
\usepackage{epsf}
\usepackage{wrapfig}
\usepackage{euscript}
\usepackage{indentfirst}
\usepackage[english,russian]{babel}
\usepackage{textcomp}
\newtheorem{theorem}{Theorem}
\newtheorem{lemma}{Lemma}
\newtheorem{definition}{Definition}
\newtheorem{proposition}{Proposition}

\newtheorem{corollary}{Corollary}
\newcommand{\co}{{\hspace{0.25mm}:\hspace{0.25mm}}}
\begin{document}
{\selectlanguage{english}
\binoppenalty = 10000 %
\relpenalty   = 10000 %

\pagestyle{headings} \makeatletter
\renewcommand{\@evenhead}{\raisebox{0pt}[\headheight][0pt]{\vbox{\hbox to\textwidth{\thepage\hfill \strut {\small Grigory. K. Olkhovikov}}\hrule}}}
\renewcommand{\@oddhead}{\raisebox{0pt}[\headheight][0pt]{\vbox{\hbox to\textwidth{{Completeness for explicit jstit logic}\hfill \strut\thepage}\hrule}}}
\makeatother

\title{Explicit justification stit logic: a completeness result}
\author{Grigory K. Olkhovikov\\Ruhr University, Bochum \\
Department of Philosophy II; Room NAFO 02/299 \\
Universit\"atsstra{\ss}e 150 \\
D--44780 Bochum, Germany \\
Tel.: +4915123707030\\
Email: grigory.olkhovikov@\{rub.de, gmail.com\}}
\date{}
\maketitle
\begin{quote}
\textbf{Abstract}. We consider the explicit fragment of the basic
justification stit logic introduced in \cite{OLWA}. We define a
Hilbert-style axiomatic system for this logic and show that this
system is strongly complete relative to the intended semantics.
\end{quote}

\begin{quote}
stit logic, justification logic, completeness, compactness
\end{quote}

\section{Introduction}
Basic justification stit (or jstit, for short) logic was
introduced in \cite{OLWA} as an environment for analysis of
doxastic actions related to proving activity within a somewhat
idealized community of agents. This logic combines expressive
means of stit logic by N. Belnap et al. \cite{belnap2001facing}
with those of justification logic by S. Artemov et al.
\cite{ArtemovN05}. The two latter logics provide for the pure
agency side and the pure proof ontology side of the proving
activity, respectively, so that it is assumed that doing something
is in effect seeing to it that something is the case, and that
every actual proof can be understood as a realization of some
proof polynomial from justification logic. The only missing
element in this picture is then the link between the two
components, i.e. how agents can see to it that a proof is
realized. Such a realization may come in different forms,
researchers may, for example, exchange emails or put the proofs
they have found on a common whiteboard. In basic jstit logic this
rather common situation is idealized in that only public proving
activity of agents is taken into account. In other words, taking
up the whiteboard metaphor, the agents in question can only
participate in proving activity by putting their proofs on the
common whiteboard for everyone to see, and not by sending one
another private messages or scribbling in their private notebooks.

In order to represent the proving activity of agents under the
above-described set of assumptions, basic jstit logic features a
set of four new modalities which were called in \cite{OLWA}
\emph{proving modalities}. Proving modalities capture four
different modes in which one can speak about proving activity of
an agent. The idea is that one gets a right classification of such
modes if one intersects the distinction between agentive and
factual (aka moment-determinate) events developed in stit logic
with the distinction between explicit and implicit modes of
knowledge which is central to justification logic. The first
distinction, when applied to proofs, corresponds to a well-known
philosophical discussion of proofs-as-objects vs proofs-as-acts.
One refers to a proof-as-act when one says that agent $j$ proves
some proposition $A$, but one refers to a proof-as-object when
saying that $A$ was proved. While doing that, one can either
simply say that $A$ was proved, or add that $A$ was proved by some
proof $t$; and the difference between these two modes of speaking
is exactly the difference between implicit and explicit reference
to proofs. The resulting classification of proving modalities
looks then as follows:

\begin{center}
 \begin{tabular}{|c||c|c|}
 \hline
  & Agentive & Moment-determinate \\
 \hline\hline
 Explicit & $j$ proves $A$ by $t$ & $A$ has been proven by $t$ \\
 & $Prove(j, t, A)$& $Proven(t,A)$\\
 \hline
 Implicit & $j$ proves $A$ & $A$ has been proven \\
 &$Prove(j, A)$& $Proven(A)$\\
 \hline
\end{tabular}
\end{center}

In \cite{OLWA} the semantics of these modalities was presented and
informally motivated in some detail. However, axiomatizing basic
jstit logic proved to be an uphill task. The first partial success
was achieved in \cite{OLimpl} where an axiomatization of implicit
fragment of basic jstit logic (obtained by omitting the two
explicit proving modalities of the above table) was presented and
shown to be complete. In the present paper we focus on this
omitted set of proving modalities and look into what happens when
one omits the implicit proving modalities from the basic jstit
logic and keeps the explicit ones. The resulting system, which
will be called here the \emph{explicit jstit logic}, complements,
in a sense, the implicit jstit logic to the full set of proving
modalities. In this paper we will axiomatize this logic and thus
complement the main result of \cite{OLimpl} with a similar result
for the explicit proving modalities.

The layout of the rest of the paper is then as follows. In Section
\ref{basic} we define the language and the semantics of the logic
at hand. We then connect explicit jstit logic to JA-STIT, the stit
logic of justification announcements. JA-STIT was the subject of
some of our earlier papers (see \cite{OLexpl-gen}) and is a proper
extension of explicit jstit logic in terms of expressive power.
The proof of the main result of the current paper displays some
very close parallels to the completeness proof for JA-STIT given
in \cite{OLexpl-gen}, to the point that we can re-use a dozen of
technical lemmas proven in \cite{OLexpl-gen} without altering one
letter in their respective proofs. We also give some facts about
expressivity of explicit jstit logic, namely, we mention the
failure finite model properties and show that this logic, just as
JA-STIT, can tell the difference between the general class of its
intended models and the subclass of models based on discrete time.

The strongly complete axiomatization for explicit jstit logic is
then presented in Section \ref{axioms}. We immediately show this
system to be sound w.r.t. the semantics introduced in Section
\ref{basic}, and we end the section with a proof for a number of
theorems of the system and a note on an alternative axiomatization
of the same set of theorems.

Section \ref{canonicalmodel} then contains the bulk of technical
work necessary for the completeness theorem. It gives a stepwise
construction and adequacy check for all the numerous components of
the canonical model and ends with a proof of a truth lemma. This
section displays the highest degree of dependency on lemmas proved
in \cite{OLexpl-gen} and we give a table connecting the lemmas of
this section to the results of \cite{OLexpl-gen}. Section
\ref{main} then wraps up, giving a concise proof of the
completeness result, drawing some conclusions and drafting
directions for future work.

In what follows we will be assuming, due to space limitations, a
basic acquaintance with both stit logic and justification logic.
We recommend to peruse \cite[Ch. 2]{horty2001agency} for a quick
introduction to the basics of stit logic, and
\cite{sep-logic-justification} for the same w.r.t. justification
logic.

\section{Basic definitions and notation}\label{basic}

\subsection{Language}

We fix some preliminaries. First we choose a finite set $Ag$
disjoint from all the other sets to be defined below. Individual
agents from this set will be denoted by letters $i$ and $j$. Then
we fix countably infinite sets $PVar$ of proof variables (denoted
by $x,y,z$) and $PConst$ of proof constants (denoted by $c,d$).
When needed, subscripts and superscripts will be used with the
above notations or any other notations to be introduced in this
paper. Set $Pol$ of proof polynomials is then defined by the
following BNF:
$$
t := x \mid c \mid s + t \mid s \times t \mid !t,
$$
with $x \in PVar$, $c \in PConst$, and $s,t$ ranging over elements
of $Pol$. In the above definition $+$ stands for the \emph{sum} of
proofs, $\times$ denotes \emph{application} of its left argument
to the right one, and $!$ denotes the so-called
\emph{proof-checker}, so that $!t$ checks the correctness of proof
$t$.

In order to define the set $Form^{Ag}$ of formulas we fix a
countably infinite set $Var$ of propositional variables to be
denoted by letters $p,q$. Formulas themselves will be denoted by
letters $A,B,C,D$, and the definition of $Form^{Ag}$ is supplied
by the following BNF:
\begin{align*}
A := p \mid A \wedge B \mid \neg A \mid [j]A \mid \Box A \mid t\co
A \mid KA \mid Prove(j,t,A) \mid Proven(t,A),
\end{align*}
with $p \in Var$, $j \in Ag$ and $t \in Pol$.

It is clear from the above definition of $Form^{Ag}$ that we are
considering a version of modal propositional language. As for the
informal interpretations of modalities, $[j]A$ is the so-called
cstit action modality and $\Box$ is the historical necessity
modality, both modalities are borrowed from stit logic. The next
two modalities, $KA$ and $t\co A$, come from justification logic
and the latter is interpreted as ``$t$ proves $A$'', whereas the
former is the strong epistemic modality ``$A$ is known''.

We assume $\Diamond$ and $\langle j \rangle$ as notations for the
dual modalities of $\Box$ and $[j]$, respectively. As usual,
$\omega$ will denote the set of natural numbers including $0$,
ordered in the natural way.

\subsection{Semantics}

For the language at hand, we assume the following semantics. A
justification stit (jstit, for short) model for a given agent
community $Ag$ is a structure of the form:
$$
\mathcal{M} = \langle Tree, \unlhd, Choice, Act, R, R_e,
\mathcal{E}, V\rangle
$$
such that:
\begin{enumerate}
\item $Tree$ is a non-empty set. Elements of $Tree$ are called
\emph{moments}.

\item $\unlhd$ is a partial order on $Tree$ for which a temporal
interpretation is assumed. We will also freely use notations like
$\unrhd$, $\lhd$, and $\rhd$ to denote the inverse relation and
the irreflexive companions.\footnote{A more common notation $\leq$
is not convenient for us since we also widely use $\leq$ in this
paper to denote the natural order relation between elements of
$\omega$.}

\item $Hist$ is a set of maximal chains in $Tree$ w.r.t. $\unlhd$.
Since $Hist$ is completely determined by $Tree$ and $\unlhd$, it
is not included into the structure of model as a separate
component. Elements of $Hist$ are called \emph{histories}. The set
of histories containing a given moment $m$ will be denoted $H_m$.
The following set:
$$
MH(\mathcal{M}) = \{ (m,h)\mid m \in Tree,\, h \in H_m \},
$$
called the set of \emph{moment-history pairs}, will be used to
evaluate formulas of the above language.

\item $Choice$ is a function mapping $Tree \times Ag$ into
$2^{2^{Hist}}$ in such a way that for any given $j \in Ag$ and $m
\in Tree$ we have as $Choice(m,j)$ (to be denoted as $Choice^m_j$
below) a partition of $H_m$. For a given $h \in H_m$ we will
denote by $Choice^m_j(h)$ the element of partition $Choice^m_j$
containing $h$.

\item $Act$ is a function mapping $MH(\mathcal{M})$ into
$2^{Pol}$.

\item $R$ and $R_e$ are two pre-order on $Tree$ giving two
versions of epistemic accessibility relation. They are assumed to
be connected by inclusion $R \subseteq R_e$.

\item $\mathcal{E}$ is a function mapping $Tree \times Pol$ into
$2^{Form^{Ag}}$ called \emph{admissible evidence function}.

\item $V$ is an evaluation function, mapping the set $Var$ into
$2^{MH(\mathcal{M})}$.
\end{enumerate}

A structure of the above described type is admitted as a jstit
model iff it satisfies the following list of additional
constraints. In order to facilitate their exposition, we introduce
a couple of useful notations first. For a given $m \in Tree$ and
any given $h, g \in H_m$ we stipulate that:
$$
Act_m := \bigcap_{h \in H_m}Act(m,h);\qquad\quad Act_{(m,h,j)} :=
\bigcap_{g \in Choice^m_j(h)}Act(m,g);
$$
and:
$$
h \approx_m g \Leftrightarrow (\exists m' \rhd m)(h, g \in
H_{m'}).
$$
Whenever we have $h \approx_m g$, we say that $h$ and $g$ are
\emph{undivided} at $m$.

We then demand satisfaction of the following constraints by the
jstit models:

\begin{enumerate}
\item Historical connection:
$$
(\forall m,m_1 \in Tree)(\exists m_2 \in Tree)(m_2 \unlhd m \wedge
m_2 \unlhd m_1).
$$

\item No backward branching:
$$
(\forall m,m_1,m_2 \in Tree)((m_1 \unlhd m \wedge m_2 \unlhd m)
\Rightarrow (m_1 \unlhd m_2 \vee m_2 \unlhd m_1)).
$$

\item No choice between undivided histories:
$$
(\forall m \in Tree)(\forall h,h' \in H_m)(h \approx_m h'
\Rightarrow Choice^m_j(h) = Choice^m_j(h'))
$$
for every $j \in Ag$.

\item Independence of agents:
$$
(\forall m\in Tree)(\forall f:Ag \to 2^{H_m})((\forall j \in
Ag)(f(j) \in Choice^m_j) \Rightarrow \bigcap_{j \in Ag}f(j) \neq
\emptyset).
$$

\item Monotonicity of evidence:
$$
(\forall t \in Pol)(\forall m,m' \in Tree)(R_e(m,m') \Rightarrow
\mathcal{E}(m,t) \subseteq \mathcal{E}(m',t)).
$$

\item Evidence closure properties. For arbitrary $m \in Tree$,
$s,t \in Pol$ and $A, B \in Form^{Ag}$ it is assumed that:
\begin{enumerate}
\item $A \to B \in \mathcal{E}(m,s) \wedge A \in \mathcal{E}(m,t)
\Rightarrow B \in \mathcal{E}(m,s\times t)$;

\item $\mathcal{E}(m,s) \cup \mathcal{E}(m,t) \subseteq
\mathcal{E}(m,s + t)$. \item $A \in \mathcal{E}(m,t) \Rightarrow
t:A \in \mathcal{E}(m,!t)$;
\end{enumerate}

\item Expansion of presented proofs:
$$
(\forall m,m' \in Tree)(m' \lhd m \Rightarrow \forall h \in H_m
(Act(m',h) \subseteq Act(m,h))).
$$

\item No new proofs guaranteed:
$$
(\forall m \in Tree)(Act_m \subseteq \bigcup_{m' \lhd m, h \in
H_m}(Act(m',h))).
$$

\item Presenting a new proof makes histories divide:
$$
(\forall m \in Tree)(\forall h,h' \in H_m)(h \approx_m h'
\Rightarrow (Act(m,h) = Act(m,h'))).
$$

\item Future always matters:
$$
\unlhd \subseteq R.
$$

\item Presented proofs are epistemically transparent:
$$
(\forall m,m' \in Tree)(R_e(m,m') \Rightarrow (Act_m \subseteq
Act_{m'})).
$$
\end{enumerate}

We offer some intuitive explanation for the above-defined notion
of jstit model. Due to space limitations, we only explain the
intuitions behind jstit models very briefly, and we urge the
reader to consult \cite[Section 3]{OLWA} for a more comprehensive
explanations, whenever needed.

The components like $Tree$, $\unlhd$, $Choice$ and $V$ are
inherited from stit logic, whereas $R$, $R_e$ and $\mathcal{E}$
come from justification logic. The only new component is $Act$
which represents the above-mentioned common pool of proofs
demonstrated to the community at any given moment under a given
history. When interpreting $Act$, we invoke the classical stit
distinction between dynamic (agentive) and static
(moment-determinate) entities, assuming that the presence of a
given proof polynomial $t$ on the community whiteboard only
becomes an accomplished fact at $m$ when $t$ is present in
$Act(m,h)$ for \emph{every} $h \in H_m$. On the other hand, if $t$
is in $Act(m,h)$ only for \emph{some} $h \in H_m$ this means that
$t$ is rather in a dynamic state of \emph{being presented}, rather
than being present, to the community.

The numbered list of semantical constraints above then just builds
on these intuitions. Constraints $1$--$4$ are borrowed from stit
logic, constraints $5$ and $6$ are inherited from justification
logic. Constraint $7$ just says that nothing gets erased from the
whiteboard, constraint $8$ says a new proof cannot spring into
existence as a static (i.e. moment-determinate) feature of the
environment out of nothing, but rather has to come as a result (or
a by-product) of a previous activity. Constraint $9$ is just a
corollary to constraint $3$ in the richer environment of jstit
models, constraint $10$ says that the possible future of the given
moment is always epistemically relevant in this moment, and
constraint $11$ says that the community knows everything that has
firmly made its way onto the whiteboard.

Right away we establish some elementary facts about jstit models
to be used in what follows:
\begin{lemma}\label{technical}
Let $\mathcal{M} = \langle Tree, \unlhd, Choice, Act, R, R_e,
\mathcal{E}, V\rangle$ be a jstit model. Then:

\begin{enumerate}
\item $(\forall m \in Tree)(\forall h \in H_m)(\{ m_1 \in Tree
\mid m_1 \unlhd m \} \subseteq h)$;

\item $(\forall m \in Tree)(\forall h,g \in H_m)(h \neq g
\Rightarrow (\exists m'' \rhd m)(h \in H_{m''}))$;

\item $(\forall m, m' \in Tree)(m \unlhd m' \Rightarrow H_{m'}
\subseteq H_m)$.
\end{enumerate}
\end{lemma}
\begin{proof}
(Part 1). We clearly have $m \in h$. Consider an arbitrary $m_1
\lhd m$. Then $h \cup \{m_1\}$ must be an $\unlhd$-chain. Indeed,
if $m' \in h$ then either $m \unlhd m'$ or $m'\lhd m$. In the
former case we get $m_1 \unlhd m'$ by transitivity of $\unlhd$, in
the latter case we get $m_1 \unlhd m' \vee m' \unlhd m_1$ by the
absence of backward branching. But since $h$ is a maximal chain,
this means that we must have $m_1 \in h$.

(Part 2). To obtain a contradiction, assume that $h,g \in H_m$ are
different, but we have:
\begin{equation}\label{E:eq}
(\forall m'' \rhd m)(h \notin H_{m''}).
\end{equation}
Given that every two elements of $h$ must be $\unlhd$-comparable,
this means that $h \subseteq \{ m_1 \in Tree \mid m_1 \unlhd m \}$
and, by Part 1, that  $h = \{ m_1 \in Tree \mid m_1 \unlhd m \}$.
Note that Part 1 also entails that $g \supseteq \{ m_1 \in Tree
\mid m_1 \unlhd m \}$, so that in this case we must have $g
\supseteq h$. We can have neither $g \supset h$, since this would
violate the maximality of $h$, nor $g = h$, since this is in
contradiction with our assumption. Therefore, \eqref{E:eq} must be
false.

(Part 3). Immediately by the absence of backward branching.
\end{proof}

For the members of $Form^{Ag}$, we will assume the following
inductively defined satisfaction relation. For every jstit model
$\mathcal{M} = \langle Tree, \unlhd, Choice, Act, R, R_e,
\mathcal{E}, V\rangle$ and for every $(m,h) \in MH(\mathcal{M})$
we stipulate that:
\begin{align*}
&\mathcal{M}, m, h \models p \Leftrightarrow (m,h) \in
V(p);\\
&\mathcal{M}, m, h \models [j]A \Leftrightarrow (\forall h'
\in Choice^m_j(h))(\mathcal{M}, m, h' \models A);\\
&\mathcal{M}, m, h \models \Box A \Leftrightarrow (\forall h'
\in H_m)(\mathcal{M}, m, h' \models A);\\
&\mathcal{M}, m, h \models KA \Leftrightarrow \forall m'\forall
h'(R(m,m') \& h' \in H_{m'} \Rightarrow \mathcal{M}, m', h'
\models A);\\
&\mathcal{M}, m, h \models t\co A \Leftrightarrow A \in
\mathcal{E}(m,t) \& (\forall m' \in Tree)(\forall h' \in H_{m'})
(R_e(m,m') \Rightarrow \mathcal{M}, m', h'
\models A);\\
&\mathcal{M}, m, h \models Prove(j, t, A) \Leftrightarrow t \in
Act_{(m,h,j)} \& \mathcal{M},m,h \models t\co A \& t \notin
Act_m;\\
&\mathcal{M}, m, h \models Proven(t, A) \Leftrightarrow t \in
Act_m \& \mathcal{M}, m, h \models t\co A
\end{align*}
In the above clauses we assume that $p \in Var$; we also assume
standard clauses for Boolean connectives.

Explicit jstit logic is closely connected to JA-STIT, the stit
logic of justification announcement. In JA-STIT the two proving
modalities of explicit jstit logic are replaced with a single
modality $Et$ for $t \in Pol$, with the following semantics:
$$
\mathcal{M}, m, h \models Et \Leftrightarrow t \in Act(m,h).
$$

Since JA-STIT is interpreted over the same class of models as
basic jstit logic, it turns out that one can retrieve the explicit
proving modalities in JA-STIT using the following definitions:
$$
Prove(j,t,A) =_{df} [j]Et \wedge \Diamond\neg Et \wedge t\co
A;\qquad\quad Proven(t,A) =_{df} \Box Et \wedge t\co A.
$$
On the other hand, it is easy to show that $Et$ cannot be defined
in explicit jstit logic, so that JA-STIT is its proper extension.
Despite the difference in expressive powers, it was possible to
borrow many constructions and lemmas for the main result of this
paper directly from the completeness proof for JA-STIT without any
modifications at all.

We observe that even though all elements of $Form^{Ag}$ are
interpreted over moment-history pairs, for some of them their
evaluations are obviously independent from the history component:
\begin{lemma}\label{determinate}
For every agent community $Ag$, every $A \in Form^{Ag}$ and every
$t \in Pol$, all of the formulas $\Box A$, $KA$, $t\co A$ and
$Proven(t, A)$ are moment-determinate, that is to say, if $\alpha
\in \{ \Box A, KA, t\co A,Proven(A) \}$, then for an arbitrary
jstit model

\noindent$\mathcal{M} = \langle Tree, \unlhd, Choice, Act, R,R_e,
\mathcal{E}, V\rangle$, arbitrary $m \in Tree$ and $h, h' \in H_m$
we have:
$$
\mathcal{M},m,h \models \alpha \Leftrightarrow \mathcal{M},m,h'
\models \alpha.
$$
Also, Boolean combinations of these formulas are
moment-determinate.
\end{lemma}
\begin{proof}
For $\alpha \in \{ \Box A, KA, t\co A\}$ it suffices to note that
their respective satisfaction conditions at a given $(m,h) \in
MH(\mathcal{M})$ in a given $\mathcal{M}$ have no free occurrences
of $h$. When we turn, further, to the corresponding condition for
$Proven(t, A)$, the only free occurrence of $h$ will be within the
context $\mathcal{M}, m, h \models t\co A$ which was shown to be
moment-determinate.

Of course, Boolean combinations of moment-determinate formulas
must be moment-determinate, too.
\end{proof}

It follows from Lemma \ref{determinate} that we might as well omit
the histories when discussing satisfaction of such formulas and
write $\mathcal{M}, m \models KA$ instead of $\mathcal{M}, m, h
\models KA$, etc.

One can in principle simplify the above semantics by introducing
the additional constraint that $R_e \subseteq R$. This leads to a
collapse of the two epistemic accessibility relation into one.
Therefore, we will call jstit models satisfying $R_e \subseteq R$
\emph{unirelational jstit models}. It is known that such a
simplification in the context of pure justification logic does not
affect the set of theorems (see, e.g. \cite{fittingreport} and
\cite[Comment 6.5]{ArtemovN05}), and we have shown in
\cite{OLexpl-gen} that this is also the case for JA-STIT. The main
result of this paper will show that also in this respect the
explicit jstit logic follows the suit. In fact, the canonical
model to be constructed in our completeness proof below is
unirelational. In view of this, we offer some comments as to the
simplifications of semantics available in the unirelational
setting.

We observe that one can equivalently define a unirelational jstit
model as a structure $\mathcal{M} = \langle Tree, \unlhd, Choice,
Act, R, \mathcal{E}, V\rangle$ satisfying all the constraints for
the jstit models, except that in the numbered constraints one
substitutes $R$ for $R_e$. Also, in the context of unirelational
jstit models, it is possible to simplify the satisfation clause
for $t\co A$ as follows:
$$
    \mathcal{M}, m, h \models t\co A \Leftrightarrow A \in
\mathcal{E}(m,t) \& \mathcal{M}, m, h \models KA.
$$

\subsection{Concluding remarks}
Before we start with the task of axiomatizing the explicit jstit
logic, we briefly mention some facts about its expressive powers
which are relevant to our chosen format of completeness proof.
Firstly, it is worth noting that under the presented semantics
some satisfiable formulas cannot be satisfied over finite models,
or even over infinite models where all histories are finite. The
argument for this is the same as in implicit fragment of basic
jstit logic, for which this claim was proved in \cite{OLimpl}
using $K(\Diamond p \wedge \Diamond\neg p)$ as an example of a
formula which is satisfiable over jstit models but not over jstit
models with finite histories. This already rules out some methods
of proving completeness like filtration method.

Second, we mention that, just like JA-STIT, explicit jstit logic
has enough expressive power to tell the difference between the
general class of jstit models and its subclass of jstit models
based on discrete time. To be more precise, we define that a jstit
model $\mathcal{M}$ is \emph{based on discrete time} iff every
chain in $Hist(\mathcal{M})$ is isomorphic to an initial segment
of $\omega$, the set of natural numbers. Then it can be shown
that:

\begin{proposition}\label{proposition}
Let $Ag$ be an agent community. The subset of
$Form^{Ag}$-validities over the class of (unirelational) jstit
models for $Ag$ is a \textbf{proper} subset of the set of
$Form^{Ag}$-validities over the class of (unirelational) jstit
models for the same community based on discrete time.
\end{proposition}
\begin{proof}
We fix an arbitrary agent community $Ag$. We clearly have the
subset relation. As for the properness part, consider the
following formula:
\begin{align*}
A := K(\neg Proven(x,p) \vee Proven(y,q)) \to (&\neg Prove(j,x,p)
\vee \\
&\vee(y\co q \to (Proven(y,q) \vee Prove(j,y,q))),
\end{align*}
 with $x, y \in PVar$, $p \in Var$, and $j \in Ag$.
 We show that $A$ is not valid over the class of all unirelational jstit models (hence not valid
over the class of all jstit models either). Consider the following
unirelational model $\mathcal{M} = \langle Tree, \unlhd, Choice,
Act, R, \mathcal{E}, V\rangle$ for the community $Ag$:
\begin{itemize}
\item $Tree = \{ a, -1 \} \cup \{ r \in \mathbb{R} \mid 0 \leq r <
1 \}$;

\item $\unlhd = \{ (0, a), (-1,a), (a,a) \} \cup \{ (r, r') \mid
r, r' \in \mathbb{R} \cap Tree, r \leq r' \}$;

\item $R = \unlhd$;

\item $\mathcal{E}(m, t) = Form^{Ag}$, for all $m \in Tree$ and $t
\in Pol$.

\item $V(p) = V(q) = MH(\mathcal{M})$, $V(p') = \emptyset$ for all
$p' \in Var \setminus \{ p,q \}$.
\end{itemize}
It is straightforward to see that the above-defined components of
$\mathcal{M}$ satisfy all the constraints imposed on normal jstit
models except possibly those involving $Choice$ and
$Act$.\footnote{It is also easy to see that $\mathcal{M}$ is
$\mathcal{CS}$-normal for any possible constant specification
$\mathcal{CS}$ as defined in the next section. Therefore,
Proposition \ref{proposition} persists when one restricts
attention to a class of jstit models normal relative to a constant
specification.} Note that, among other things, we will have, under
the above settings, that:
\begin{equation}\label{E:prop1}
    \mathcal{M}, m \models x\co p \wedge y\co q
\end{equation}
for all $m \in Tree$. Before we go on and define the remaining
components, let us pause a bit and reflect on the structure of
histories in the model $\mathcal{M}$ that is being defined. We
only have two histories in it, one is $h_1 = \{ -1, 0, a \}$ and
the other is $h_2 = \{ -1 \} \cup \{ r \in \mathbb{R} \mid 0 \leq
r < 1 \}$. So we define:
\begin{align*}
    Choice^m_i = \left\{%
\begin{array}{ll}
    H_m, & \hbox{if $i \neq j$ or $m \neq 0$;} \\
    \{\{h_1\},\{h_2\}\}, & \hbox{if $i = j$ and $m = 0$.} \\
\end{array}%
\right.
\end{align*}
\begin{align*}
    Act(m,h) = \left\{%
\begin{array}{ll}
    \{ x,y \}, & \hbox{if $m \in \mathbb{R}$ and $m > 0$;} \\
    \{x \}, & \hbox{if $m = 0$ and $h = h_2$;} \\
    \emptyset, & \hbox{otherwise.} \\
\end{array}%
\right.
\end{align*}
Again, most of the constraints on jstit models are now easily seen
to be satisfied. The no new proofs guaranteed constraint is
perhaps less straightforward, so we consider it in some detail. We
have, on the one hand, $Act_m = \emptyset$, whenever $m \in \{ -1,
0, a \}$ so neither of these moments can falsify the constraint.
The only remaining option is that $m \in \{ r \in \mathbb{R} \mid
0 < r < 1 \}$, say $m = r$. But then the only history passing
through $r$ is $h_2$ and we have, on the other hand, $\frac{r}{2}
\in Tree$, $\frac{r}{2} < r$, and $Act(\frac{r}{2}, h_2) =
Act(r,h_2) = Act_r$ so that the no new proofs guaranteed
constraint is again verified.

Now, consider $0 \in Tree$. The set of its epistemic alternatives
is $Tree \setminus \{ -1 \}$. We have all of the following:
$Choice^0_j(h_2) = \{h_2\}$, $x \in Act(0,h_2)$,  $H_0 = \{ h_1,
h_2\}$, $x \notin Act(0,h_1)$, and $y \notin Act(0,h_2)$. In
virtue of all this and by \eqref{E:prop1}, we obtain that:
\begin{equation}\label{E:prop2}
\mathcal{M}, 0, h_2 \models  Prove(j,x,p) \wedge y\co q \wedge
\neg Prove(j,y,q) \wedge \neg Proven(y,q).
\end{equation}
On the other hand, if $m \in Tree$ and $h \in H_m$, then either $m
\in \{ r \in \mathbb{R} \mid 0 < r < 1 \}$ or $m$ is outside of
this set. In the former case we have $H_m = \{h_2 \}$, whence
$Act_m = \{x,y\}$ so that, by \eqref{E:prop1}, we get:
$$
\mathcal{M}, m,h \models  Proven(y,q),
$$
for all $h \in H_m$. In the latter case we have $Act_m =
\emptyset$, which means that we must also have:
$$
\mathcal{M}, m,h \models  \neg Proven(x,p),
$$
for all $h \in H_m$. So, in any case, the formula $\neg
Proven(x,p) \vee Proven(y,q)$ gets to be satisfied throughout all
of the moment-history pairs in $\mathcal{M}$, which further means
that:
$$
\mathcal{M}, 0, h_2 \models  K(\neg Prove(j,x,p) \vee
Prove(j,y,q))
$$
is satisfied. Adding the latter with \eqref{E:prop2}, we see that
$(0,h_2)$ falsifies $A$ in $\mathcal{M}$.

On the other hand, $A$ is valid in the class of jstit models based
on discrete time (hence also over unirelational jstit models based
on discrete time). In order to show this, we will assume its
invalidity and obtain a contradiction. Indeed, let

\noindent$\mathcal{M} = \langle Tree, \unlhd, Choice, Act, R,R_e,
\mathcal{E}, V\rangle$ be a jstit model based on discrete time
such that $\mathcal{M}, m, h \not\models A$.  Then we will have
both
\begin{equation}\label{E:discr1}
\mathcal{M}, m, h \models  K(\neg Proven(x,p) \vee Proven(y,q)),
\end{equation}
and
\begin{equation}\label{E:discr2}
\mathcal{M}, m, h \models Prove(j,x,p) \wedge y\co q \wedge \neg
Prove(j,y,q) \wedge \neg Proven(y,q).
\end{equation}
Note that the failure of $Proven(y,q)$ combined with satisfaction
of $y\co q$ shows that we cannot have $y \in Act_m$. On the other
hand, the failure of $Prove(j,y,q)$ at $(m,h)$ leaves us with the
following options:
$$
(\mathcal{M}, m, h \not\models y\co q) \vee y \in Act_m \vee
(\exists g \in Choice^m_j(h))(y \notin Act(m,g)).
$$
Thus we know that for some $h' \in Choice^m_j(h)$ we will have $y
\notin Act(m,h')$. Also, note that for any such $h'$ we will have
$Choice^m_j(h') = Choice^m_j(h)$. Adding this up with the
satisfaction of $Prove(j,x,p)$ at $(m,h)$, we get that one can
choose an $h' \in Choice^m_j(h)$ in such a way that the following
holds:
\begin{equation}\label{E:discr2-2}
y \notin Act(m,h') \& (\mathcal{M}, m, h' \models Prove(j,x,p)
\wedge y\co q).
\end{equation}
By Lemma \ref{determinate} and  \eqref{E:discr1}, we also know
that for such $h'$ we will have:
\begin{equation}\label{E:discr1-1}
\mathcal{M}, m, h' \models  K(\neg Proven(x,p) \vee Proven(y,q)),
\end{equation}
Next, we observe that since $(m,h')$ satisfies $Prove(j,x,p)$, we
know that $x \in Act(m,h')$ and also that there is a $g \in H_m$
such that $x \notin Act(m,g)$. This shows that we must have $h'
\neq g$, and, by Lemma \ref{technical}.2, this means $m$ must have
$\lhd$-successors along $h'$. Since $\mathcal{M}$ is based on
discrete time, consider embedding $f$ of $h'$ into an initial
segment of $\omega$. Suppose that $f(m) = n$. We have established
that $m$ is not the $\unlhd$-last moment along $h'$, so there must
be an $m' \in h'$ such that $f(m') = n + 1$. By the embedding
properties of $f$, this means that $m \lhd m'$ and for no $m'' \in
Tree$ it is true that $m \lhd m'' \lhd m'$. By the future always
matters constraint, we know that $R(m,m')$, therefore, by
\eqref{E:discr1-1} we must have:
\begin{equation}\label{E:discr5}
\mathcal{M}, m', h' \models \neg Proven(x,p) \vee Proven(y,q).
\end{equation}
On the other hand, let $g \in H_{m'}$ be arbitrary. Then, by Lemma
\ref{technical}.3, $g \in H_m$, and, moreover, $g \approx_m h'$.
Therefore, by the presenting a new proof makes histories divide
constraint, we must have $Act(m, g) = Act(m,h')$. By
\eqref{E:discr2-2} we know that $x \in Act(m,h')$, which means
that also  $x \in Act(m,g)$. Since $g \in H_{m'}$ was chosen
arbitrarily, the latter means that $x \in \bigcap_{g \in
H_{m'}}(Act(m,g))$, and, by the expansion of presented proofs
constraint, $x \in Act_{m'}$. Further, it follows from
\eqref{E:discr2-2} that:
\begin{equation}\label{E:discr6}
\mathcal{M}, m, h' \models x\co p.
\end{equation}
Given that $R \subseteq R_e$, we must have $R_e(m,m')$, whence by
the monotonicity of evidence and the pre-order properties of $R_e$
we further obtain that:
\begin{equation}\label{E:discr7}
\mathcal{M}, m', h' \models x\co p.
\end{equation}
Since we know that $x \in Act_{m'}$, \eqref{E:discr7} immediately
leads to:
\begin{equation}\label{E:discr9}
\mathcal{M}, m', h' \models Proven(x,p).
\end{equation}
Whence, in view of \eqref{E:discr5}, it follows that
\begin{equation}\label{E:discr10}
\mathcal{M}, m', h' \models Proven(y,q).
\end{equation}
The latter means that $y \in Act_{m'}$, and by the no new proofs
guaranteed constraint, it follows that for some $g \in H_{m'}$ and
some $m'' \in g$ such that $m'' \lhd m'$, we must have $y \in
Act(m'',g)$. Now, if $m'' \lhd m'$, then $m'' \unlhd m$, since
$m'$ was chosen as the immediate $\lhd$-successor of $m$ along
$h'$. The latter means, by the expansion of presented proofs, that
$y \in Act(m,g)$. Since, as we have shown above, $g \approx_m h'$,
this means, by the presenting a new proof makes histories divide
constraint, that $Act(m, g) = Act(m,h')$ and, further, that $y \in
Act(m,h')$. The latter is in obvious contradiction with
\eqref{E:discr2-2}.

The obtained contradiction shows that $A$ is valid over the class
of jstit models based on discrete time, so that it must also be
valid over its unirelational subclass.
\end{proof}

\section{Axiomatic system and soundness}\label{axioms}

Throughout this section, and the next one, we are going to let
$Ag$ serve as a constant denoting arbitrary but fixed agent
community. We consider the Hilbert-style axiomatic system $\Pi$
with the following set of axiomatic schemes:
\begin{align}
&\textup{A full set of axioms for classical propositional
logic}\label{A0}\tag{\text{A0}}\\
&\textup{$S5$ axioms for $\Box$ and $[j]$ for every $j \in
Ag$}\label{A1}\tag{\text{A1}}\\
&\Box A \to [j]A \textup{ for every }j \in Ag\label{A2}\tag{\text{A2}}\\
&(\Diamond[j_1]A_1 \wedge\ldots \wedge \Diamond[j_n]A_n) \to
\Diamond([j_1]A_1 \wedge\ldots \wedge[j_n]A_n)\label{A3}\tag{\text{A3}}\\
&(s\co(A \to B) \to (t\co A \to (s\times t)\co
B)\label{A4}\tag{\text{A4}}\\
&t\co A \to (!t\co(t\co A) \wedge KA)\label{A5}\tag{\text{A5}}\\
&(s\co A \vee t\co A) \to (s+t)\co A\label{A6}\tag{\text{A6}}\\
&\textup{$S4$ axioms for $K$}\label{A7}\tag{\text{A7}}\\
&KA \to \Box K\Box A\label{A8}\tag{\text{A8}}\\
&Prove(j, t, A) \to (\neg Proven(t, A) \wedge [j]Prove(j, t, A) \wedge \neg\Box Prove(j,t,A)\wedge t\co A)\label{B9}\tag{\text{B9}}\\
&(Prove(j, t, A) \wedge t\co B) \to Prove(j, t, B)\label{B10}\tag{\text{B10}}\\
&Proven(t,A) \to (KProven(t,A) \wedge t\co A)\label{B11}\tag{\text{B11}}\\
&(Proven(t, A) \wedge t\co B) \to Proven(t, B)\label{B12}\tag{\text{B12}}\\
&\neg Prove(j,t,A) \to \langle j \rangle (\bigwedge_{i \in Ag}\neg
Prove(i,t,A))\label{B13}\tag{\text{B13}}
\end{align}

The assumption is that in \eqref{A3} $j_1,\ldots, j_n$ are
pairwise different.

To this set of axiom schemes we add the following rules of
inference:
\begin{align}
&\textup{From }A, A \to B \textup{ infer } B;\label{R1}\tag{\text{R1}}\\
&\textup{From }A\textup{ infer }KA;\label{R2}\tag{\text{R2}}\\
&\textup{From }KA \to (\neg Proven(t_1, B_1) \vee\ldots \vee\neg
Proven(t_n, B_n))\notag\\
&\textup{      infer } KA \to (\bigwedge_{j \in Ag}\neg
Prove(j,t_1, B_1) \vee\ldots \vee \bigwedge_{j \in Ag}\neg
Prove(j,t_n, B_n)).\label{S4}\tag{\text{S4}}
\end{align}
The different notation styles present in the above sets of axioms
and inference rules are meant to underscore that the axioms
\eqref{A0}--\eqref{A8} and rules \eqref{R1}, \eqref{R2} are shared
by $\Pi$ with other axiomatizations for logics combining
justification and stit modalities, including the axiomatization of
the implicit jstit logic given in \cite{OLimpl} and the
axiomatization of JA-STIT given in \cite{OLexpl-gen}.

A standard way to obtain extensions of $\Pi$ is by adding to it
\emph{constant specifications}, which basically ensure that one
has enough pre-assigned proofs for the axioms of this system. More
precisely, a constant specification is a set $\mathcal{CS}$ such
that:
\begin{itemize}
\item $\mathcal{CS} \subseteq \{ c_n\co\ldots c_1\co A\mid
c_1,\ldots,c_n \in PConst, A\textup{ an instance of
}\eqref{A0}-\eqref{A8}, \eqref{B9}-\eqref{B13} \}$;

\item Whenever $c_{n+1}\co c_n\co\ldots c_1\co A \in
\mathcal{CS}$, then also $c_n\co\ldots c_1\co A \in \mathcal{CS}$.
\end{itemize}
A given constant specification can be added to $\Pi$ by appending
the following inference rule \eqref{RCS} to its set of rules:
\begin{align}
\textup{If }c_n\co\ldots c_1\co A \in \mathcal{CS},\textup{ infer
} c_n\co\ldots c_1\co A.\label{RCS}\tag{\text{$R_{\mathcal{CS}}$}}
\end{align}
The resulting axiomatic system is then called $\Pi(\mathcal{CS})$.
Note that $\emptyset$ is clearly one example of constant
specification and that we have $\Pi(\emptyset) = \Pi$. Whenever
$\mathcal{CS} \neq \emptyset$, the system $\Pi(\mathcal{CS})$ ends
up proving some formulas which are not valid over the class of
jstit models. Nevertheless, restriction on jstit models which
comes with a commitment to a given $\mathcal{CS}$ is relatively
straightforward to describe. We say that a jstit model
$\mathcal{M}$ is $\mathcal{CS}$\emph{-normal} iff it is true that:
\begin{align*}
(\forall c \in PConst)(\forall m \in Tree)(\{ A \mid c\co A\in
\mathcal{CS} \} \subseteq \mathcal{E}(m,c)),
\end{align*}
where $\mathcal{E}$ is the $\mathcal{M}$'s admissible evidence
function. Again, it is easy to see that the class of
$\emptyset$-normal jstit models is just the whole class of jstit
models so that the representation $\Pi(\emptyset) = \Pi$ does not
place any additional restrictions on the class of intended models
of $\Pi$.

For a given constant specification $\mathcal{CS}$, we define that
a \emph{proof} in $\Pi(\mathcal{CS})$ as a finite sequence of
formulas such that every formula in it is either an instance of
one of the schemes \eqref{A0}--\eqref{A8}, \eqref{B9}--\eqref{B12}
or is obtained from earlier elements of the sequence by one of the
inference rules \eqref{R1},\eqref{R2}, \eqref{RCS}, \eqref{S4}. A
proof is a proof of its last formula. If an $A \in Form^{Ag}$ is
provable in $\Pi(\mathcal{CS})$, we will write
$\vdash_{\mathcal{CS}} A$. We say that $\Gamma \subseteq
Form^{Ag}$ is $\mathcal{CS}$-\emph{inconsistent} iff for some
$A_1,\ldots,A_n \in \Gamma$ we have $\vdash_{\mathcal{CS}} (A_1
\wedge\ldots \wedge A_n) \to \bot$, and we say that $\Gamma$ is
$\mathcal{CS}$-consistent iff it is not
$\mathcal{CS}$-inconsistent. $\Gamma$ is
$\mathcal{CS}$-\emph{maxiconsistent} iff it is
$\mathcal{CS}$-consistent and no $\mathcal{CS}$-consistent subset
of $Form^{Ag}$ properly extends $\Gamma$.

We observe that this definition allows for the standard operations
with consistent and maxiconsistent sets. Namely, every
$\mathcal{CS}$-consistent set $\Gamma$ can be extended to a
$\mathcal{CS}$-maxiconsistent set $\Delta \supseteq \Gamma$, and
$\mathcal{CS}$-maxiconsistent sets are regular relative to the
propositional connectives in that for every
$\mathcal{CS}$-maxiconsistent set $\Gamma$ and all $A,B \in
Form^{Ag}$ all of the following holds:
\begin{itemize}
\item Exactly one element of $\{A, \neg A \}$ is in $\Gamma$.

\item $A \vee B \in \Gamma$ iff $(A \in \Gamma$ or $B \in
\Gamma)$.

\item If $A, (A \to B) \in \Gamma$, then $B \in \Gamma$.

\item $A \wedge B \in \Gamma$ iff $(A \in \Gamma$ and $B \in
\Gamma)$.
\end{itemize}

Our goal is now to obtain, for any given constant specification
$\mathcal{CS}$, a strong completeness theorem for
$\Pi(\mathcal{CS})$, and we start by establishing some soundness
claims:

\begin{theorem}\label{soundness}
Let $\mathcal{CS}$ be an arbitrary constant specification and let
$A \in Form^{Ag}$ be such that $\vdash_{\mathcal{CS}} A$. Then $A$
is valid over the class of $\mathcal{CS}$-normal jstit models.
\end{theorem}
\begin{proof}
Given the above notion of proof, it is sufficient to show that
every instance of \eqref{A0}--\eqref{A8}, \eqref{B9}--\eqref{B12}
is valid over the class of $\mathcal{CS}$-normal jstit models and
that every application of rules \eqref{R1}, \eqref{R2},
\eqref{RCS}, and \eqref{S4} to formulas which are valid over the
class of $\mathcal{CS}$-normal jstit models yields a formula which
is valid over the class of $\mathcal{CS}$-normal jstit models.

First, note that if $\mathcal{M} = \langle Tree, \unlhd, Choice,
Act, R, R_e, \mathcal{E}, V\rangle$ is a normal jstit model, then
$\langle Tree, \unlhd, Choice, V\rangle$ is a model of stit logic.
Therefore, axioms \eqref{A0}--\eqref{A3}, which were copy-pasted
from the standard axiomatization of \emph{dstit}
logic\footnote{See, e.g. \cite[Ch. 17]{belnap2001facing}, although
$\Pi$ uses a simpler format closer to that given in \cite[Section
2.3]{balbiani}.} must be valid. Second, note that if $\mathcal{M}
= \langle Tree, \unlhd, Choice, Act, R, R_e, \mathcal{E},
V\rangle$ is a normal jstit model, then $\mathcal{M} = \langle
Tree, R, R_e, \mathcal{E}, V\rangle$ is what is called in
\cite[Section 6]{ArtemovN05} a justification model\footnote{The
format for the variable assignment $V$ is slightly different, but
this is of no consequence for the present setting.}. This means
that also all of the \eqref{A4}--\eqref{A7} must be valid, given
that all of them were borrowed from the standard axiomatization of
justification logic. The validity of other axioms will be
motivated below in some detail. In what follows, $\mathcal{M} =
\langle Tree, \unlhd, Choice, Act, R, R_e,\mathcal{E}, V\rangle$
will always stand for an arbitrary $\mathcal{CS}$-normal jstit
model, and $(m,h)$ for an arbitrary element of $MH(\mathcal{M})$.

As for \eqref{A8}, assume for \emph{reductio} that $\mathcal{M},
m,h \models KA \wedge \Diamond K \Diamond\neg A$. Then (using
Lemma \ref{determinate} to omit the histories) $\mathcal{M}, m
\models KA$ and also $\mathcal{M}, m \models K \Diamond\neg A$. By
reflexivity of $R$, it follows that $\Diamond\neg A$ will be
satisfied at $m$ in $\mathcal{M}$. The latter means that, for some
$h' \in H_m$, $A$ must fail at $(m,h')$ and therefore, again by
reflexivity of $R$, $KA$ must fail at $m$ in $\mathcal{M}$, a
contradiction.

Consider \eqref{B9} and assume that $\mathcal{M}, m,h \models
Prove(j,t,A)$. Then $t \notin Act_m$, which immediately implies
that:
\begin{equation}\label{E:cs1}
\mathcal{M}, m,h \not\models Proven(t,A).
\end{equation}
Next, we must have, just by the satisfaction of $Prove(j,t,A)$,
that:
\begin{equation}\label{E:cs2}
\mathcal{M}, m \models t\co A.
\end{equation}
Further, note that for every $g \in Choice^m_j(h)$ we will have
$Choice^m_j(g) = Choice^m_j(h)$, so that for every such $g$ we
will have $t \in Act_{(m,g,j)}$. Adding this up with \eqref{E:cs2}
and the fact that $t \notin Act_m$, we get that $Prove(j,t,A)$ is
satisfied at $(m,g)$ for every $g \in Choice^m_j(h)$, or, in other
words, that we have:
\begin{equation}\label{E:cs3}
\mathcal{M}, m,h \models [j]Prove(j,t,A).
\end{equation}
Finally, given $t \notin Act_m$, consider $h' \in H_m$ such that
$t \notin Act(m,h')$. Given that $h' \in Choice^m_j(h')$, we know
that $t \notin Act_{(m,h',j)}$, which means that $Prove(j,t,A)$
fails at $(m,h')$ so that,in view of $h' \in H_m$, we get:
\begin{equation}\label{E:cs4}
\mathcal{M}, m,h \models \neg\Box Prove(j,t,A).
\end{equation}
Summing up \eqref{E:cs1}--\eqref{E:cs4}, we see that \eqref{B9} is
satisfied at $(m,h)$.

As for \eqref{B10}, assume that $Prove(j,t,A) \wedge t\co B$ is
satisfied at $(m,h)$. By the satisfaction of the first conjunct we
get that  $t \in Act_{(m,h,j)} \& t \notin Act_m$, which, together
with the satisfaction of $t\co B$, yields that $\mathcal{M}, m,h
\models Prove(j,t,B)$.

Next we consider \eqref{B11}. Assuming that $\mathcal{M}, m,h
\models Proven(t,A)$, we immediately get that:
\begin{equation}\label{E:cs5}
t \in Act_m
\end{equation}
and that:
\begin{equation}\label{E:cs6}
\mathcal{M}, m\models t\co A.
\end{equation}
Assume that $m' \in Tree$ is such that $R(m,m')$. Then we must
have:
\begin{equation}\label{E:cs7}
\mathcal{M}, m'\models t\co A
\end{equation}
by \eqref{E:cs5} and $R \subseteq R_e$, and we must also have:
\begin{equation}\label{E:cs8}
t \in Act_{m'}
\end{equation}
by the epistemic transparency of presented proofs. Thus we will
have $M, m' \models Proven(t,A)$ for an arbitrary $R$-successor
$m'$ of $m$, which means, by definition, that we will have:
\begin{equation}\label{E:cs7}
\mathcal{M}, m\models KProven(t,A).
\end{equation}
Taken together, \eqref{E:cs6} and \eqref{E:cs7} show that
\eqref{B11} is satisfied at $(m,h)$.

As for \eqref{B12}, assume that $Proven(t,A) \wedge t\co B$ is
satisfied at $m$. By the satisfaction of the first conjunct we get
that  $t \in Act_m$, which together with the satisfaction of $t\co
B$ yields that $\mathcal{M}, m,h \models Proven(t,B)$.

The last axiomatic scheme is \eqref{B13}. Assume that
$Prove(j,t,A)$ fails at $(m,h)$. This can happen for different
reasons, therefore, we have to distinguish between three cases:

\emph{Case 1}. $\mathcal{M}, m,h \not\models t\co A$. Then, by the
validity of \eqref{B9} we must have

\noindent$\mathcal{M}, m,h \models \bigwedge_{i \in Ag}\neg
Prove(i,t,A)$. Given that $h \in Choice^m_j(h)$, we further obtain
that $\mathcal{M}, m,h \models \langle j\rangle\bigwedge_{i \in
Ag}\neg Prove(i,t,A)$, and the axiom is satisfied.

\emph{Case 2}. $\mathcal{M}, m,h \models t\co A$ and $t \in
Act_m$. Then we must have $\mathcal{M}, m,h \models Proven(t,A)$
and, again by \eqref{B9}, we must have $\mathcal{M}, m,h \models
\bigwedge_{i \in Ag}\neg Prove(i,t,A)$. Reasoning as in Case 1, we
again see that the axiom is satisfied.

\emph{Case 3}. $\mathcal{M}, m,h \models t\co A$ and $t \notin
Act_m$. But then, given the failure of $Prove(j,t,A)$ at $(m,h)$,
there must be some $h' \in Choice^m_j(h)$ such that $t \notin
Act(m,h')$. Notice that we will have then $h' \in Choice^m_i(h')$
for every $i \in Ag$. Therefore, for every $i \in Ag$ we will have
$t \notin Act_{(m,h',i)}$, whence it follows that $\mathcal{M},
m,h' \models \bigwedge_{i \in Ag}\neg Prove(i,t,A)$. Since $h' \in
Choice^m_j(h)$, this further means that $\mathcal{M}, m,h \models
\langle j\rangle\bigwedge_{i \in Ag}\neg Prove(i,t,A)$, and the
axiom is satisfied.

 Taking up the rules of inference, we immediately see that
 \eqref{R1} and \eqref{R2} can only return
 $\mathcal{CS}$-validities when given another
 $\mathcal{CS}$-validities as premises. As for \eqref{RCS}, assume
 that $B = c_{n+1}\co c_n\co\ldots c_1\co A \in \mathcal{CS}$. We
 argue by induction on $n \geq 0$.

 \emph{Basis}. If $n = 0$ then $B = c_1\co A \in \mathcal{CS}$.
 Since $\mathcal{M}$ is $\mathcal{CS}$-normal, this means that $A
 \in \mathcal{E}(m,c_1)$. Also $A$ must be an instance of one of
 the above axiomatic schemes which were all shown to be
 $\mathcal{CS}$-validities above, which means that $A$ must hold
 at every moment-history pair in $\mathcal{M}$, including the
 pairs where the moment is some $R_e$-successor of $m$. Therefore,
 we must have $\mathcal{M}, m,h \models c_1\co A = B$.

 \emph{Induction step}. $n = k + 1$. Then $B = c_{k+2}\co c_{k+1}\co\ldots c_1\co A \in \mathcal{CS}$.
 By definition of constant specifications, we will also have then
  $c_{k+1}\co\ldots c_1\co A \in \mathcal{CS}$. By induction
  hypothesis, we know that  $c_{k+1}\co\ldots c_1\co A$ is a
  $\mathcal{CS}$-validity, hence must hold in every moment-history
  pair of $\mathcal{M}$, including those
 pairs where the moment is some $R_e$-successor of $m$. By
 $\mathcal{CS}$-normality of $\mathcal{M}$ we also know that  $c_{k+1}\co\ldots c_1\co A \in
 \mathcal{E}(m,c_{k+2})$, which shows that $\mathcal{M}, m,h \models
 B$.

The hardest part is to show the soundness of the rule \eqref{S4}.
Assume that

\noindent$KA \to (\neg Proven(t_1, B_1) \vee\ldots \vee\neg
Proven(t_n, B_n))$ is valid over $\mathcal{CS}$-normal jstit
models, and assume also that for some $\mathcal{CS}$-normal jstit
model $\mathcal{M} = \langle Tree, \unlhd, Choice, Act, R, R_e,
\mathcal{E}, V\rangle$ and for some $(m,h)\in MH(\mathcal{M})$ we
have:
\begin{equation}\label{E:e1}
\mathcal{M}, m, h \models KA \wedge (\bigvee_{j \in Ag}
Prove(j,t_1, B_1) \wedge\ldots \wedge \bigvee_{j \in Ag}
Prove(j,t_n, B_n)).
\end{equation}
Then we can choose $j_1,\ldots,j_n$ in such a way that we have:
\begin{equation}\label{E:e2}
\mathcal{M}, m, h \models KA \wedge (Prove(j_1,t_1, B_1)
\wedge\ldots \wedge Prove(j_n,t_n, B_n)).
\end{equation}
By the definition of satisfaction relation, we obtain that:
\begin{equation}\label{E:e3}
\mathcal{M}, m, h \models t_1\co B_1 \wedge\ldots \wedge t_n\co
B_n.
\end{equation}
The latter basically means two things:
\begin{equation}\label{E:e4}
 B_1 \in \mathcal{E}(m,t_1),\ldots, B_n \in \mathcal{E}(m,t_n).
\end{equation}
and
\begin{equation}\label{E:e5}
 (\forall m_0 \in Tree)(\forall h_0 \in H_{m_0})(R_e(m,m_0) \Rightarrow \mathcal{M},m_0,h_0 \models B_1\wedge\ldots\wedge B_n).
\end{equation}
On the other hand, we obtain from \eqref{E:e2}, also by the
definition of satisfaction relation and $h \in Choice^m_j(h)$,
that:
\begin{equation}\label{E:e5-1}
t_1,\ldots, t_n \in Act(m,h).
\end{equation}
We also know that we can choose a $g \in H_m$ such that $t_1
\notin Act(m,g)$. This means that $h \neq g$. By Lemma
\ref{technical}.2, it follows that we can choose an $m' \in h$
such that $m' \rhd m$. So we choose such an $m'$. By Lemma
\ref{technical}.3 $H_{m'} \subseteq H_m$, and, moreover, every
history in $H_{m'}$ is undivided from $h$ at $m$. By the
presenting a new proof makes histories divide constraint, this
means that:
\begin{equation}\label{E:e6}
(\forall g \in H_{m'})(Act(m,g) = Act(m,h)).
\end{equation}
By \eqref{E:e5-1} and \eqref{E:e6}, this means that:
\begin{equation}\label{E:e7}
t_1,\ldots, t_n \in \bigcap_{g \in H_{m'}}Act(m,g).
\end{equation}
Note that it follows from $m' \rhd m$ and the expansion of
presented proofs constraint that $\bigcap_{g \in H_{m'}}Act(m,g)
\subseteq Act_{m'}$, so that we must have, by \eqref{E:e7}, that:
\begin{equation}\label{E:e8}
t_1,\ldots, t_n \in Act_{m'}.
\end{equation}
Next, it follows from \eqref{E:e4} by the monotonicity of evidence
that:
\begin{equation}\label{E:e9}
 B_1 \in \mathcal{E}(m',t_1),\ldots, B_n \in \mathcal{E}(m',t_n),
\end{equation}
and it follows from $m' \rhd m$ by the future always matters
constraint and the inclusion $R \subseteq R_e$ that $R_e(m,m')$.
From the latter fact we get, by \eqref{E:e5} and transitivity of
$R_e$ that:
\begin{equation}\label{E:e10}
 (\forall m_0 \in Tree)(\forall h_0 \in H_{m_0})(R_e(m',m_0) \Rightarrow \mathcal{M},m_0,h_0 \models B_1\wedge\ldots\wedge B_n).
\end{equation}
In their turn, \eqref{E:e9} and \eqref{E:e10} yield that:
\begin{equation}\label{E:e3}
\mathcal{M}, m' \models t_1\co B_1 \wedge\ldots \wedge t_n\co B_n,
\end{equation}
by the definition of satisfaction relation. Adding this up with
\eqref{E:e8} we get that:
\begin{equation}\label{E:e11}
\mathcal{M}, m', h \models Proven(t_1, B_1) \wedge\ldots \wedge
Prove(t_n, B_n).
\end{equation}
Finally, by $m' \rhd m$ and the future always matters constraint,
we get that $R(m,m')$, whence, by transitivity of $R$ and
\eqref{E:e2}, we obtain that:
\begin{equation}\label{E:e12}
\mathcal{M}, m', h \models KA.
\end{equation}
Taken together, \eqref{E:e11} and \eqref{E:e12} contradict the
assumed validity of

\noindent$KA \to (\neg Proven(t_1, B_1) \vee\ldots \vee\neg
Proven(t_n, B_n))$, which shows that \eqref{E:e1} cannot be true
for any moment-history pair in any $\mathcal{CS}$-normal jstit
model.
\end{proof}

Before treating completeness, we make some elementary observations
about provability in the systems of the form $\Pi(\mathcal{CS})$.
We first state some theorems and derivable rules of $\Pi$.
\begin{lemma}\label{p-theorems}
Let $A \in Form^{Ag}$, $j, i_1,\ldots, i_n \in Ag$ and $t \in
Pol$. Then all of the following theorems and derived rules are
provable in $\Pi$:
\begin{align}
&KA \to \Box A\label{T0}\tag{\text{T0}}\\
&\textup{From }A\textup{ infer }\Box A\label{R'1}\tag{\text{R'1}}\\
&\textup{From }A\textup{ infer }[j]A\label{R'2}\tag{\text{R'2}}\\
&t\co A \to Kt\co A\label{T1}\tag{\text{T1}}\\
&t\co A \to \Box t\co A\label{T2}\tag{\text{T2}}\\
&KA \to \Box KA\label{T3}\tag{\text{T3}}\\
&Proven(t,A) \to \Box Proven(t,A)\label{T4}\tag{\text{T4}}\\
&\neg\Box(Prove(i_1,t,A)\vee\ldots\vee
Prove(i_n,t,A))\label{T5}\tag{\text{T5}}
\end{align}
\end{lemma}
\begin{proof}
\eqref{T0}. We use the transitivity of implication w.r.t. the
following set of formulas:
\begin{align}
    &KA \to \Box K\Box A\label{E:z3} &&\text{(by \eqref{A8})}\\
    &\Box K\Box A \to K\Box A\label{E:z4} &&\text{(by \eqref{A1})}\\
    &K\Box A \to \Box A\label{E:z5} &&\text{(by \eqref{A7})}
\end{align}
\eqref{R'1}. From $A$ we infer $KA$ by \eqref{R2} and then use
\eqref{T0} and modus ponens to get $\Box A$.

\eqref{R'2}. From $A$ we infer $\Box A$ by \eqref{R'1} and then
apply \eqref{A2} and modus ponens.

We pause to note that by \eqref{R'1} and \eqref{R'2} we know that
every modality in the set $\{\Box\} \cup \{ [j] \mid j \in Ag \}$
is an S5-modality.

\eqref{T1}. We have both $t\co A \to !t\co(t\co A)$ and
$!t\co(t\co A) \to Kt\co A$ by \eqref{A5} so that we get
\eqref{T1} by transitivity of implication.

\eqref{T2}. By \eqref{T1} and \eqref{T0}.

\eqref{T3}. By $KA \to KKA$ (a part of \eqref{A7}) and \eqref{T0}.

\eqref{T4}. By \eqref{B11} and \eqref{T0}.

\eqref{T5}. We first prove the theorem for $n = 1$. In this case,
note that we have $\Box Prove(i_1,t,A) \to Prove(i_1,t,A)$ by
\eqref{A1} whence by contraposition we get

\noindent$\neg Prove(i_1,t,A) \to \neg\Box Prove(i_1,t,A)$. We
also have $Prove(i_1,t,A) \to \neg\Box Prove(i_1,t,A)$ by
\eqref{B9}. By classical propositional logic we get then:
$$
(Prove(i_1,t,A) \vee \neg Prove(i_1,t,A)) \to \neg\Box
Prove(i_1,t,A),
$$
and, further $\neg\Box Prove(i_1,t,A)$, as desired. We now turn to
the general case and sketch the derivation as follows:
\begin{align}
    &\neg Prove(i_1,t,A) \to \langle i_1\rangle(\bigwedge_{j \in
    Ag}\neg Prove(j,t,A))\label{E:z6} &&\text{(by \eqref{B13})}\\
    &\langle i_1\rangle(\bigwedge_{j \in Ag}\neg Prove(j,t,A))
    \to \langle i_1\rangle(\bigwedge^n_{k = 1}\neg Prove(i_k,t,A))\label{E:z7} &&\text{($[i_1]$ is S5)}\\
    &\neg Prove(i_1,t,A) \to\langle i_1\rangle(\bigwedge^n_{k = 1}\neg Prove(i_k,t,A))\label{E:z8} &&\text{(by \eqref{E:z6} and \eqref{E:z7})}\\
    &[i_1](\bigvee^n_{k = 1}Prove(i_k,t,A)) \to Prove(i_1,t,A)\label{E:z9} &&\text{(by \eqref{E:z8}, contrap.)}\\
    &\Box(\bigvee^n_{k = 1}Prove(i_k,t,A)) \to Prove(i_1,t,A)\label{E:z10} &&\text{(by \eqref{E:z9}, \eqref{A2})}\\
    &\Box(\bigvee^n_{k = 1}Prove(i_k,t,A)) \to \Box Prove(i_1,t,A)\label{E:z11} &&\text{($\Box$ is S5)}
\end{align}
From \eqref{E:z11}, \eqref{T5} follows by the case for $n = 1$ and
classical propositional logic.
\end{proof}
Our second point is that the rule \eqref{S4} can be substituted by
an infinite array of axiomatic schemes without affecting the set
of provable formulas, which gives us, in effect, an alternative
axiomatization for the systems of the form $\Pi(\mathcal{CS})$.
More precisely, the following lemma holds:
\begin{lemma}\label{equivalent}
Let $\mathcal{CS}$ be a constant specification. Consider the
following axiomatic scheme:
\begin{align}
    &K(\neg Proven(t_1,B_1) \vee\ldots\vee \neg Proven(t_n,B_n)) \to\notag\\
    &\qquad\qquad\qquad\qquad\to(\bigwedge_{j \in Ag}\neg
Prove(j,t_1, B_1) \vee\ldots \vee \bigwedge_{j \in Ag}\neg
Prove(j,t_n, B_n))\label{AS}\tag{\text{$A_{S4}$}}
\end{align}
for arbitrary $t_1,\ldots, t_n \in Pol$ and $B_1,\ldots, B_n \in
Form^{Ag}$. Let $\Pi'(\mathcal{CS})$ be the axiomatic system
obtained from $\Pi(\mathcal{CS})$ by replacing \eqref{S4} with
\eqref{AS}. Then, for every $A \in Form^{Ag}$, it is true that $A$
is provable in $\Pi(\mathcal{CS})$ iff $A$ is provable in
$\Pi'(\mathcal{CS})$
\end{lemma}
\begin{proof} ($\Leftarrow$) Assume that $t_1,\ldots, t_n \in Pol$ and $B_1,\ldots, B_n \in
Form^{Ag}$. Then, setting:
$$
B:=\neg Proven(t_1,B_1) \vee\ldots\vee \neg Proven(t_n,B_n),
$$
it suffices to note that the respective instance of \eqref{AS} is
provable in $\Pi(\mathcal{CS})$ by one application of \eqref{S4}
to $KB \to B$, which itself is an axiom by \eqref{A7}.

($\Rightarrow$). Assume that $t_1,\ldots, t_n \in Pol$ and
$B_1,\ldots, B_n \in Form^{Ag}$. Then, let $B$ be as above and
set:
$$
C := \bigwedge_{j \in Ag}\neg Prove(j,t_1, B_1) \vee\ldots \vee
\bigwedge_{j \in Ag}\neg Prove(j,t_n, B_n).
$$
Then note that every transition from $KD \to B$ to $KD \to C$
according to \eqref{S4} can be replaced with a proof in
$\Pi'(\mathcal{CS})$ sketched below:
\begin{align}
    &KD \to B\label{E:pr1} &&\textup{(premise)}\\
    &K(KD \to B)\label{E:pr2} &&\textup{(by \eqref{E:pr1} and \eqref{R2})}\\
    &KD \to KB\label{E:pr3} &&\textup{(by \eqref{E:pr2}, \eqref{A7}, and \eqref{R1})}\\
    &KB \to C\label{E:pr4} &&\textup{(by \eqref{AS})}\\
    &KD \to C\label{E:pr5} &&\textup{(by \eqref{E:pr3}, \eqref{E:pr4}, and \eqref{A0})}
\end{align}
\end{proof}

We are now prepared to formulate our main result:

\begin{theorem}\label{completeness} Let $\Gamma \subseteq
Form^{Ag}$ and let $\mathcal{CS}$ be a constant specification.
Then $\Gamma$ is $\mathcal{CS}$-consistent iff it is satisfiable
in a $\mathcal{CS}$-normal jstit model iff it is satisfiable in a
unirelational $\mathcal{CS}$-normal jstit model.
\end{theorem}

One part of the completeness results we have, of course, right
away, as a consequence of Theorem \ref{soundness}:

\begin{corollary}\label{c-soundness}
Let $\Gamma \subseteq Form^{Ag}$ and let $\mathcal{CS}$ be a
constant specification. If $\Gamma$ is satisfiable in a
$\mathcal{CS}$-normal (unirelational) jstit model, then $\Gamma$
is $\mathcal{CS}$-consistent.
\end{corollary}
\begin{proof}
Let $\Gamma \subseteq Form^{Ag}$ be satisfiable in a
$\mathcal{CS}$-normal jstit model $\mathcal{M}$, either
unirelational or not. Then for some $(m,h) \in MH(\mathcal{M})$ we
have $\mathcal{M}, m, h \models \Gamma$. If $\Gamma$ were
$\mathcal{CS}$-inconsistent this would mean that for some
$A_1,\ldots,A_n \in \Gamma$ we would have $\vdash_{\mathcal{CS}}
(A_1 \wedge\ldots \wedge A_n) \to \bot$. By Theorem
\ref{soundness}, this would mean that:
$$
\mathcal{M}, m, h \models (A_1 \wedge\ldots \wedge A_n) \to \bot,
$$
whence clearly $\mathcal{M}, m, h \models \bot$, which is
impossible. Therefore, $\Gamma$ must be $\mathcal{CS}$-consistent.
\end{proof}

\section{The canonical model}\label{canonicalmodel}

We begin by fixing an arbitrary constant specification
$\mathcal{CS}$ throughout the present section. The main aim of the
section is to prepare the proof of the inverse of Corollary
\ref{c-soundness}. The method used is a variant of the canonical
model technique, but, due to the complexity of the case, we do not
define our model in one sweeping definition. Rather, we proceed
piecewise, defining elements of the model one by one, and checking
the relevant constraints as soon, as we have got enough parts of
the model in place. The last subsection proves the truth lemma for
the defined model. As we have already indicated, the model to be
built will be a normal unirelational jstit model, so that $R_e$
will be omitted, or, equivalently, assumed to coincide with $R$.
It should also be noted that our definitions of stit- and
justifications-related components of the canonical model are
borrowed to the last letter from the construction of the canonical
model for JA-STIT given in \cite{OLexpl-gen}. Even though the
basic building blocks for our current case are somewhat different
from those used for JA-STIT case, this does not affect the proofs
of the respective lemmas in the least. Therefore, we omit the
proofs of almost every lemma claimed in subsection \ref{borrowed}
and replace them with the following table bringing the lemmas in
question into correspondence with the respective lemmas of
\cite{OLexpl-gen} so that the reader may look up the proofs
whenever this is called for.

\begin{center}
 \begin{tabular}{|c|c|}
 \hline
  Numbering given in subsection \ref{borrowed}  & Reference to \cite{OLexpl-gen} \\
 \hline
 Lemma \ref{leq} &  Lemma 4\\
 Lemma \ref{injective}& Lemma 5\\
Lemma \ref{structure} &  Lemma 9\\
 Lemma \ref{sqcap}& Lemma 10\\
 Lemma \ref{choice}& Lemma 12\\
 Lemma \ref{r}& Lemma 14\\
 Lemma \ref{e}& Lemma 15\\
 \hline
\end{tabular}
\end{center}

The canonical model to be constructed below will be a
$\mathcal{CS}$-normal jstit model named
$\mathcal{M}^{Ag}_{\mathcal{CS}}$. The ultimate building blocks of
$\mathcal{M}^{Ag}_{\mathcal{CS}}$ we will call \emph{elements}.
Before going on with the definition of
$\mathcal{M}^{Ag}_{\mathcal{CS}}$, we define what these elements
are and explore some of their properties.

\begin{definition}\label{element}
An \emph{element} is a sequence of the form
$(\Gamma_1,\ldots,\Gamma_n)$ for some $n\in \omega$ with $n \geq
1$ such that:
\begin{itemize}
\item For every $k \leq n$, $\Gamma_k$ is a
$\mathcal{CS}$-maxiconsistent subset of $Form^{Ag}$;

\item For every $k < n$, for all $A \in Form^{Ag}$, if $KA \in
\Gamma_k$, then $KA \in \Gamma_{k + 1}$;

\item For every $k < n$, for all $t \in Pol$, $A \in Form^{Ag}$,
and $j \in Ag$, if $ Prove(j,t,A) \in \Gamma_k$, then $Proven(t,A)
\in \Gamma_{k + 1}$.
\end{itemize}
\end{definition}

We prove the following lemma:
\begin{lemma}\label{elementcontinuation}
Whenever $(\Gamma_1,\ldots,\Gamma_n)$ is an element, there exists
a $\Gamma_{n + 1} \subseteq Form^{Ag}$ such that the sequence
$(\Gamma_1,\ldots,\Gamma_{n + 1})$ is also an element.
\end{lemma}
\begin{proof}
Assume $(\Gamma_1,\ldots,\Gamma_n)$ is an element and consider the
following set:
$$
\Delta := \{ KA \mid KA \in \Gamma_n \} \cup \{ Proven(t,A) \mid
(\exists j \in Ag)(Prove(j,t,A) \in \Gamma_n) \}.
$$

We show that $\Delta$ is $\mathcal{CS}$-consistent. Of course, the
set $\{ KA \mid KA \in \Gamma_n \}$ is $\mathcal{CS}$-consistent
since it is a subset of $\Gamma_n$ and the latter is assumed to be
$\mathcal{CS}$-consistent. Further, if $\Delta$ is
$\mathcal{CS}$-inconsistent, then, wlog, for some $B_1,\ldots,
B_k, C_ 1,\ldots, C_l \in Form^{Ag}$, $t_1,\ldots, t_l \in Ag$,
and $j_1, \ldots, j_l \in Ag$ such that $KB_1,\ldots, KB_k$ and
$Prove(j_1,t_1, C_1),\ldots, Prove(j_l,t_l,C_l)$ are in $\Gamma_n$
we will have:
$$
\vdash_\mathcal{CS}(KB_1\wedge\ldots \wedge KB_r) \to (\neg
Proven(t_1,C_1) \vee\ldots \vee \neg Proven(t_l,C_l)),
$$
whence, by \eqref{A7}:
$$
\vdash_\mathcal{CS} K(B_1\wedge\ldots \wedge B_r) \to (\neg
Prove(t_1,C_1) \vee\ldots \vee \neg Proven(t_l,C_l)),
$$
and further, by \eqref{S4}:
$$
\vdash_\mathcal{CS} K(B_1\wedge\ldots \wedge B_r) \to
(\bigwedge_{j \in Ag}\neg Prove(j,t_1, C_1) \vee\ldots \vee
\bigwedge_{j \in Ag}\neg Prove(j,t_l, C_l)),
$$
whence, by \eqref{A0} and \eqref{R1} we get that:
$$
\vdash_\mathcal{CS} K(B_1\wedge\ldots \wedge B_r) \to (\neg
Prove(j_1,t_1, C_1) \vee\ldots \vee \neg Prove(j_l,t_l, C_l)).
$$
The latter formula shows that $\Gamma_n$ is
$\mathcal{CS}$-inconsistent which contradicts the assumption that
$(\Gamma_1,\ldots,\Gamma_n)$ is an element.

Therefore, $\Delta$ must be $\mathcal{CS}$-consistent, and is also
extendable to a $\mathcal{CS}$-maxiconsistent $\Gamma_{n + 1}$. By
the choice of $\Delta$, this means that
$(\Gamma_1,\ldots,\Gamma_n, \Gamma_{n + 1})$ must be an element.
\end{proof}
The structure of elements will be important in what follows. If
$\xi = (\Gamma_1,\ldots, \Gamma_n)$ is an element and an element
$\tau$ is of the form $(\Gamma_1,\ldots, \Gamma_{k})$ with $k <
n$, we say that $\tau$ is a \emph{proper} initial segment of
$\xi$. Moreover, if $k = n -1$, then $\tau$ is the \emph{greatest}
proper initial segment of $\xi$. We define $n$ to be  the
\emph{length} of $\xi$. Furthermore, we define that $\Gamma_n$ is
the end element of $\xi$ and write $\Gamma_n = end(\xi)$.

We now define the canonical model using elements as our building
blocks. We start by defining the following relation $\equiv$:

\begin{align*}
(\Gamma_1,\ldots, \Gamma_n, \Gamma_{n + 1}) \equiv
(\Delta_1,\ldots,& \Delta_n, \Delta_{n + 1}) \Leftrightarrow
(\Gamma_1 = \Delta_1 \& \ldots \& \Gamma_n = \Delta_n
\&\\
 &\& (\forall A \in Form^{Ag})(\Box A \in \Gamma_{n + 1} \Rightarrow A
\in \Delta_{n + 1}).
\end{align*}

It is routine to check that $\equiv$ is an equivalence relation
given that $\Box$ is an S5 modality. The notation
$[(\Gamma_1,\ldots, \Gamma_n)]_\equiv$ will denote the
$\equiv$-equivalence class generated by $(\Gamma_1,\ldots,
\Gamma_n)$. Since all the elements inside a given
$\equiv$-equvalence class are of the same length, we may extend
the notion of length to these classes setting that the length of
$[(\Gamma_1,\ldots, \Gamma_n)]_\equiv$ also equals $n$.

The next two lemmas will be repeatedly used in what follows:
\begin{lemma}\label{element}
Let $(\Gamma_1,\ldots,\Gamma_n, \Gamma)$ be an element, let
$\Delta \subseteq Form^{Ag}$ be $\mathcal{CS}$-maxiconsistent and
let:
$$
\{ \Box A \mid \Box A \in \Gamma \} \subseteq \Delta.
$$
Then $(\Gamma_1,\ldots,\Gamma_n, \Delta)$ is also an element, and,
moreover:
$$
(\Gamma_1,\ldots,\Gamma_n, \Gamma) \equiv
(\Gamma_1,\ldots,\Gamma_n, \Delta).
$$
\end{lemma}
\begin{proof}
We first show that $(\Gamma_1,\ldots,\Gamma_n, \Delta)$ is an
element. Indeed, if $KA \in \Gamma_n$, then $KA \in \Gamma$ by
definition of an element. But then $\Box KA \in \Gamma$ by
\eqref{T2} and $\mathcal{CS}$-maxiconsistency of $\Gamma$, whence
$\Box KA \in \Delta$. By \eqref{A1} and
$\mathcal{CS}$-maxiconsistency of $\Delta$ we get then $KA \in
\Delta$. Similarly, if $Prove(j,t,A) \in \Gamma_n$, then
$Proven(t,A) \in \Gamma$ by definition of an element. But then
$\Box Proven(t,A) \in \Gamma$ by \eqref{T4} and
$\mathcal{CS}$-maxiconsistency of $\Gamma$, whence $\Box
Proven(t,A) \in \Delta$.

Given the inclusion $\{ \Box A \mid \Box A \in \Gamma \} \subseteq
\Delta$, the other part of the lemma is straightforward.
\end{proof}
\begin{lemma}\label{element-new}
Let $\xi$ be an element and let $A \in Form^{Ag}$. Then the
following statements hold:
\begin{enumerate}
\item $A \in \bigcap_{\tau \equiv \xi}end(\tau) \Leftrightarrow
\Box A \in end(\xi) \Leftrightarrow \Box A \in \bigcap_{\tau
\equiv \xi}end(\tau)$.

\item If $\vdash_\mathcal{CS} A \to \Box A$, then $A \in end(\xi)
\Leftrightarrow A \in \bigcap_{\tau \equiv \xi}end(\tau)
\Leftrightarrow \Box A \in \bigcap_{\tau \equiv \xi}end(\tau)$.
\end{enumerate}
\end{lemma}
\begin{proof}
(Part 1) If $\Box A \in \bigcap_{\tau \equiv \xi}end(\tau)$, then,
of course, $\Box A \in end(\xi)$, whence, by definition of
$\equiv$ we get that $A \in \bigcap_{\tau \equiv \xi}end(\tau)$.
On the other hand, if $A \in \bigcap_{\tau \equiv \xi}end(\tau)$,
then choose an arbitrary $\tau$ such that $\tau \equiv \xi$. If
$\Box A \in end(\tau)$, then we are done. Otherwise, we can obtain
a contradiction as follows. Consider the set:
$$
\Gamma = \{ \Box B \mid \Box B \in end(\tau) \} \cup \{\neg A\}.
$$
If $\Gamma$ were $\mathcal{CS}$-inconsistent, then we would have:
$$
\vdash_\mathcal{CS}(\Box B_1 \wedge\ldots\wedge \Box B_n) \to A
$$
for some $\Box B_1,\ldots, \Box B_n \in end(\tau)$, whence by S5
reasoning for $\Box$ we would also have that:
$$
\vdash_\mathcal{CS}(\Box B_1 \wedge\ldots\wedge \Box B_n) \to \Box
A.
$$
By $\mathcal{CS}$-maxiconsistency of $end(\tau)$ it would follow
then that $\Box A \in end(\tau)$, contrary to our assumption.
Therefore, $\Gamma$ is $\mathcal{CS}$-consistent and we can extend
it to a $\mathcal{CS}$-maxiconsistent $\Delta$. Consider the inner
structure of $\tau$. We must have $\tau =
(\Gamma_1,\ldots,\Gamma_n, end(\tau))$ for appropriate $n \geq 0$
and $\Gamma_1,\ldots,\Gamma_n \subseteq Form^{Ag}$. But then, by
Lemma \ref{element}, we must also have that
$(\Gamma_1,\ldots,\Gamma_n, \Delta)$ is an element, and, moreover,
that $(\Gamma_1,\ldots,\Gamma_n, \Delta) \equiv \tau \equiv \xi$.
But then, note that by $\Gamma \subseteq \Delta$ we have $\neg A
\in \Delta$, whence by $\mathcal{CS}$-consistency $A \notin
\Delta$ in contradiction with our assumption that $A \in
\bigcap_{\tau \equiv \xi}end(\tau)$. This contradiction shows that
for no $\tau \equiv \xi$ can we have $\Box A \notin end(\tau)$ so
that we must end up having $\Box A \in \bigcap_{\tau \equiv
\xi}end(\tau)$.

(Part 2). If $\vdash_\mathcal{CS} A \to \Box A$ then, by
\eqref{A1} and $\mathcal{CS}$-maxiconsistency of $end(\xi)$, we
will have $\Box A \in end(\xi) \Leftrightarrow A \in end(\xi)$,
and the rest follows by Part 1.
\end{proof}

We now proceed to definitions of components for the canonical
model.

\subsection{Stit and justification components}\label{borrowed}

The first two components of the canonical model
$\mathcal{M}^{Ag}_{\mathcal{CS}}$ are as follows:

\begin{itemize}
    \item $Tree = \{ \dag \} \cup \{ ([\xi]_\equiv, n)\mid n \in \omega,\,\xi\textup{ is an element}\}$.
    Thus the elements of $Tree$, with the exception of the special
    moment $\dag$, are $\equiv$-equivalence classes of elements
    coupled with natural numbers. Such moments we will call
    \emph{standard} moments, and the left projection of a
    standard moment $m$ we will call its \emph{core} (and write $\overrightarrow{m}$), while the
    right projection of such moment we will call its \emph{height} (and write
    $|m|$). In this way, we get the equality $m = (\overrightarrow{m}, |m|)$
    for every standard $m \in Tree$.
    We further define that the length of a standard moment $m$ is
    the length of its core. For the sake of completeness, we extend the above notions
    to $\dag$ setting both length and height of this moment to $0$ and defining that $\overrightarrow{\dag} = \dag$.
\item We set that $(\forall m \in Tree\setminus\{ \dag \})(\dag
\lhd m \& m \ntriangleleft \dag)$. We further set that for any two
standard moments $m$ and $m'$, we have that $m \lhd m'$ iff either
(1) there exists a $\xi \in
    \overrightarrow{m}$ such that for every $\tau\in
    \overrightarrow{m'}$, $\xi$ is a proper initial segment of $\tau$,
     or (2) $ \overrightarrow{m} =  \overrightarrow{m'}$ and $|m'| < |m|$.
    The relation $\unlhd$ is then defined as the reflexive
    companion to $\lhd$.
\end{itemize}

With this settings, we claim that the restraints imposed by our
semantics on $Tree$ and $\unlhd$ are satisfied:
\begin{lemma}\label{leq}
The relation $\unlhd$, as defined above, is a partial order on
$Tree$, which satisfies both historical connection and no backward
branching constraints.
\end{lemma}

Before we move on to the other components of the canonical model
$\mathcal{M}$ to be defined in this section, we formulate some
important facts about the structure of
$Hist(\mathcal{M}^{Ag}_{\mathcal{CS}})$ as induced by the
above-defined $Tree$ and $\unlhd$. We start by defining basic
sequences of elements. A \emph{basic sequence} of elements is a
set of elements of the form $\{ \xi_1,\ldots,\xi_n,\ldots, \}$
such that for every $n \geq 1$:
\begin{itemize}
    \item $\xi_n$ is of length $n$;
    \item $\xi_n$ is the greatest proper initial segment of
    $\xi_{n + 1}$.
\end{itemize}
Basic sequences will be denoted by capital Latin letters $S$, $T$,
and $U$ with subscripts and superscripts when needed. Every given
basic sequence $S$ induces the following $[S] \subseteq Tree$:
$$
[S] = \{\dag\} \cup \bigcup_{n \in \omega}\{ ([\xi_n]_\equiv, k)
\mid k \in \omega \}.
$$
It is immediate that every basic sequence $S$ induces a unique
$[S] \subseteq Tree$ in this way. It is, perhaps, less immediate
that the mapping $S \mapsto [S]$ is injective:
\begin{lemma}\label{injective}
Let $S$, $T$ be basic sequences of elements. Then:
$$
[S] = [T] \Leftrightarrow S = T.
$$
\end{lemma}

Another striking fact is that basic sequences can be used to
characterize $Hist(\mathcal{M}^{Ag}_{\mathcal{CS}})$ through this
injection:
\begin{lemma}\label{structure}
The following statements hold:

\begin{enumerate}
\item If $S = \{ \xi_1,\ldots,\xi_n,\ldots, \}$ is a basic
sequence, then $[S] \in Hist(\mathcal{M}^{Ag}_{\mathcal{CS}})$,
and the following presentation gives $[S]$ in the
$\unlhd$-ascending order:
$$
\dag,\ldots,([\xi_1]_\equiv, k),\ldots, ([\xi_1]_\equiv,
0),\ldots,([\xi_n]_\equiv, k),\ldots, ([\xi_n]_\equiv, 0),\ldots,
$$

\item $Hist(\mathcal{M}^{Ag}_{\mathcal{CS}}) = \{ [S] \mid
S\textup{ is a basic sequence}\}$.
\end{enumerate}
\end{lemma}

It follows from Lemmas \ref{injective} and \ref{structure} that
not only every basic sequence generates a unique $h \in
Hist(\mathcal{M}^{Ag}_{\mathcal{CS}})$, but also for every $h \in
Hist(\mathcal{M}^{Ag}_{\mathcal{CS}})$ there exists a unique basic
sequence $S$ such that $h = [S]$. We will denote this unique $S$
for a given $h$ by $]h[$. It is immediate from Lemmas
\ref{injective} and \ref{structure} that for every $h \in
Hist(\mathcal{M}^{Ag}_{\mathcal{CS}})$, $h = [(]h[)]$. Likewise,
for every basic sequence $S$, we have $S = ]([S])[$. As a further
useful piece of notation, we introduce the notion of
\emph{intersection} of a standard moment $m$ with a history $h \in
H_m$. Assume that $m$ is of the length $n$ and that $]h[  = \{
\xi_1,\ldots,\xi_n,\ldots, \}$. Then $m$ must be of the form
$([\xi_n]_\equiv, k)$ for some $k\in\omega$, and we will also have
$\overrightarrow{m} \cap ]h[ = \{ \xi_n \}$. We now define the
only member of the latter singleton as the result $m \sqcap h$ of
the intersection of $m$ and $h$, setting $m \sqcap h = \xi_n$. It
can be shown that for any element $\xi$ in the core of a given
standard moment $m$ there exists an $h \in H_m$ such that $\xi = m
\sqcap h$:
\begin{lemma}\label{sqcap}
Let $(\Gamma_1,\ldots,\Gamma_k)$ be an element. Then, for every $n
\in \omega$ there is at least one history $h \in
H_{([(\Gamma_1,\ldots,\Gamma_k)]_\equiv, n)}$ such that
$([(\Gamma_1,\ldots,\Gamma_k)]_\equiv, n) \sqcap h =
(\Gamma_1,\ldots,\Gamma_k)$.
\end{lemma}

An immediate but important corollary of Lemma \ref{sqcap} is that
the core of a given moment $m$ is exactly the set of $m$'s
intersections with the histories passing through $m$:
\begin{corollary}\label{c-sqcap}
Let $m \in Tree$. Then $\{ \xi \mid \xi \in \overrightarrow{m} \}
= \{ m \sqcap h \mid h \in H_m \}$.
\end{corollary}

We offer some general remarks on what we have shown thus far.
Lemma \ref{structure} shows that every history in the canonical
model has a uniform order structure, namely, it consists of $\dag$
followed by $\omega$ copies of the set of negative integers.
Another general observation is that histories in
$\mathcal{M}^{Ag}_{\mathcal{CS}}$ can only branch off at moments
of height $0$, so that at moments of other heights all the
histories remain undivided. This last fact does not follow from
the lemmas proved thus far but it can be proved in the same way as
we have proved the similar fact in \cite[Corollary 3]{OLexpl-gen}.

We now define the choice function for our canonical model:
\begin{itemize}
    \item $Choice^m_j(h) = \{ h' \in H_m \mid
    (\forall A \in Form)([j]A \in end(h \sqcap m) \Rightarrow A \in end(h' \sqcap
    m))\}$, if $m \neq \dag$ and $|m| = 0$;
    \item $Choice^m_j = H_m$, otherwise.
\end{itemize}
Since for every $j \in Ag$, $[j]$ is an S5-modality, $Choice$
induces a partition on $H_m$ for every given $m \in Tree$. We
claim that the choice function verifies the relevant semantic
constraints:
\begin{lemma}\label{choice}
The tuple $\langle Tree, \unlhd, Choice\rangle$, as defined above,
verifies both independence of agents and no choice between
undivided histories constraints.
\end{lemma}

The next two lemmas can be viewed as `stit versions' of Lemmas
\ref{element} and \ref{element-new}.
\begin{lemma}\label{element-stit}
Let $j \in Ag$, let $(\Gamma_1,\ldots,\Gamma_n, \Gamma)$ be an
element, let $\Delta \subseteq Form^{Ag}$ be
$\mathcal{CS}$-maxiconsistent and let:
$$
\{ [j]A \mid [j]A \in \Gamma \} \subseteq \Delta.
$$
Then $(\Gamma_1,\ldots,\Gamma_n, \Delta)$ is also an element, and
for $m = ((\Gamma_1,\ldots,\Gamma_n, \Gamma)_\equiv, 0)$, whenever
$m \sqcap h = (\Gamma_1,\ldots,\Gamma_n, \Gamma)$, there exists a
$g \in Choice^m_j(h)$ such that:
$$
(\Gamma_1,\ldots,\Gamma_n, \Delta) = m \sqcap g.
$$
\end{lemma}
\begin{proof}
First of all, note that whenever $\Box A \in \Gamma$ we have, in
virtue of $\mathcal{CS}$-maxiconsistency of $\Gamma$, that
$\Box\Box A \in \Gamma$ (by \eqref{A1}) and, further, that
$[j]\Box A \in \Gamma$ (by \eqref{A2}). Therefore, we must have
$[j]\Box A \in \Delta$ and in view of
$\mathcal{CS}$-maxiconsistency of $\Delta$ and \eqref{A1} we will
also have $\Box A \in \Delta$. Thus we have established that $\{
\Box A \mid \Box A \in \Gamma \} \subseteq \Delta$ and it follows,
by Lemma \ref{element}, that $(\Gamma_1,\ldots,\Gamma_n, \Delta)$
is an element and that $(\Gamma_1,\ldots,\Gamma_n, \Delta)\in
\overrightarrow{m}$. Now, if $h$ is chosen as in the lemma, use
Lemma \ref{sqcap} to pick a $g \in H_m$ such that
$(\Gamma_1,\ldots,\Gamma_n, \Delta) = m \sqcap g$ holds. Recall
that we have $|m| = 0$. By the construction of $\Delta$ and
\eqref{A1} we must then have $g \in Choice^m_j(h)$.
\end{proof}
Before moving on with the lemmas, we introduce two further
notations, which are similar to the notations we used for $Act$,
but refer to the inner structures of
$\mathcal{M}^{Ag}_{\mathcal{CS}}$:
$$
end_m = \bigcap_{h \in H_m}end(m \sqcap h);\qquad\quad
end_{(m,h,j)} = \bigcap_{g \in Choice^m_j(h)}end(m \sqcap g).
$$
\begin{lemma}\label{element-new-stit}
Let $m \in Tree$ be such that $|m| = 0$, and let $A \in
Form^{Ag}$. Then:
$$
A \in end_{(m,h,j)} \Leftrightarrow [j] A \in end(m \sqcap h)
\Leftrightarrow [j]A \in end_{(m,h,j)}.
$$
\end{lemma}
\begin{proof}
If $[j]A \in end_{(m,h,j)}$, then, by $h \in Choice^m_j(h)$, $[j]A
\in end(m \sqcap h)$, whence, by $|m| = 0$ we get that $A \in
end_{(m,h,j)}$. On the other hand, if $A \in end_{(m,h,j)}$, then
choose an arbitrary $g \in Choice^m_j(h)$. If $[j]A \in m \sqcap
g$, then we are done. Otherwise, we can obtain a contradiction as
follows. Consider the set
$$
\Gamma = \{ [j]B \mid [j]B \in end(m \sqcap g) \} \cup \{\neg A\}.
$$
If $\Gamma$ were $\mathcal{CS}$-inconsistent, then we would have:
$$
\vdash_\mathcal{CS}([j]B_1 \wedge\ldots\wedge [j]B_n) \to A,
$$
for some $[j]B_1,\ldots, [j]B_n \in end(m \sqcap g)$, whence by S5
reasoning for $[j]$ we would also have that:
$$
\vdash_\mathcal{CS}([j]B_1 \wedge\ldots\wedge [j]B_n) \to [j]A.
$$
By $\mathcal{CS}$-maxiconsistency of $end(m \sqcap g)$ it would
follow then that $[j]A \in end(m \sqcap g)$, contrary to our
assumption. Therefore, $\Gamma$ is $\mathcal{CS}$-consistent and
we can extend it to a $\mathcal{CS}$-maxiconsistent $\Delta$.
Consider the inner structure of $m \sqcap g$. We must have $m
\sqcap g = (\Gamma_1,\ldots,\Gamma_n, end(m \sqcap g))$ for
appropriate $n \geq 0$ and $\Gamma_1,\ldots,\Gamma_n \subseteq
Form^{Ag}$. But then, by Lemma \ref{element-stit}, we must also
have that $(\Gamma_1,\ldots,\Gamma_n, \Delta)$ is an element, and,
moreover, that $(\Gamma_1,\ldots,\Gamma_n, \Delta) = m \sqcap h'$
for some $h' \in Choice^m_j(g) = Choice^m_j(h)$ (the latter
equality obtains by $g \in Choice^m_j(h)$). But then, note that by
$\Gamma \subseteq \Delta$ we have $\neg A \in \Delta$, whence by
$\mathcal{CS}$-consistency $A \notin \Delta = end(m \sqcap h')$ in
contradiction with our assumption that $A \in end_{(m,h,j)}$. This
contradiction shows that for no $g \in Choice^m_j(h)$ can we have
$[j]A \notin end(\tau)$ so that we must end up having $[j]A \in
end_{(m,h,j)}$.
\end{proof}
We sum up the implications of the above lemmas for our modalities
as follows:
\begin{corollary}\label{p-mod}
Let $m \in Tree$, $h \in H_m$, $A \in Form^{Ag}$, $t \in Pol$, and
$j \in Ag$. Then:
\begin{enumerate}
\item $\alpha \in end(m \sqcap h) \Leftrightarrow \alpha \in
end_m$, for all $\alpha \in \{ \Box A, t\co A, KA, Proven(t,A)
\}$;

\item $\alpha \in end(m \sqcap h) \Leftrightarrow \alpha \in
end_{(m,h,j)}$, for all $\alpha \in \{ [j]A, Prove(j,t,A) \}$.
\end{enumerate}
\end{corollary}
\begin{proof}
Part 1 we get by \eqref{A1}, \eqref{T2}--\eqref{T4}, Lemma
\ref{element-new}.2, and Corollary \ref{c-sqcap}. Part 2 we get by
\eqref{A1}, \eqref{B9}, $\mathcal{CS}$-maxiconsistency of end sets
of elements, and Lemma \ref{element-new-stit}.
\end{proof}

Turning to the justifications-related components, we first define
$R$ as follows:
\begin{itemize}
    \item $R(([(\Gamma_1,\ldots,\Gamma_n, \Gamma)]_\equiv, k), m')\Leftrightarrow$

     $\qquad\qquad\qquad\quad\Leftrightarrow(m' \neq \dag)\&(\forall A \in Form^{Ag})(KA \in \Gamma \Rightarrow KA \in
    end_{m'})$;
    \item $R(\dag,m)$, for all $m \in Tree$.
\end{itemize}

Now, for the definition of $\mathcal{E}$:
\begin{itemize}
    \item For
    all $t \in Pol$: $\mathcal{E}(\dag, t) = \{ A \in Form^{Ag} \mid
\vdash t\co A \}$;

\item For all $t \in Pol$ and $m \neq \dag$: $(\forall A \in
Form^{Ag})(A \in \mathcal{E}(m, t) \Leftrightarrow t\co A \in
end_m)$.

\end{itemize}
We start by mentioning a straightforward corollary to the above
definition:
\begin{lemma}\label{proven}
For all $m \in Tree$ and $t \in Pol$ it is true that $\{ A \in
Form^{Ag} \mid \vdash t\co A \} \subseteq \mathcal{E}(m,t)$.
\end{lemma}
\begin{proof}
This holds simply by the definition of $\mathcal{E}$ when $m =
\dag$. If $m \neq \dag$, then, for every $\xi \in
\overrightarrow{m}$, $end(\xi)$ is a maxiconsistent subset of
$Form^{Ag}$ and must contain every provable formula.
\end{proof}

Note that it follows from Lemma \ref{proven}, that the
above-defined function $\mathcal{E}$ satisfies the
$\mathcal{CS}$-normality condition on jstit models. We now mention
the respective adequacy claims:

\begin{lemma}\label{r}
The relation $R$, as defined above, is a preorder on $Tree$, and,
together with $\unlhd$, verifies the future always matters
constraint.
\end{lemma}

\begin{lemma}\label{e}
The function $\mathcal{E}$, as defined above, satisfies both
monotonicity of evidence and evidence closure properties.
\end{lemma}

\subsection{$Act$ and $V$}

It remains to define $Act$ and $V$ for our canonical model, and we
define them as follows:
\begin{itemize}
    \item $(m,h) \in V(p) \Leftrightarrow p \in end(m \sqcap h)$,
    for all $p \in Var$;
    \item $Act(\dag, h) = \emptyset$ for all $h \in Hist(\mathcal{M})$;
    \item $Act(m,h) = \{ t \in Pol \mid (\exists A \in Form^{Ag}, j \in Ag)(Proven(t,A) \in end(m \sqcap
    h)\vee Prove(j,t,A) \in end(m \sqcap
    h))
    \}$, if $m \neq \dag$, $|m| = 0$ and $h \in H_m$;
    \item $Act(m,h) = \{ t \in Pol \mid (\exists A \in Form^{Ag})(Proven(t,A) \in end(m \sqcap
    g))
    \}$, if $m \neq \dag$, $|m| > 0$ and $h \in H_m$
\end{itemize}

Since in the definition of $Act$ we have used the proving
modalities not available in JA-STIT, we can no longer rely on
constructions carried over for the canonical model of
\cite{OLexpl-gen}.

We first draw some of the immediate consequences of the above
definitions:
\begin{lemma}\label{act-technical}
Assume that $m \in Tree \setminus \{ \dag \}$, $h \in H_m$, and $t
\in Pol$. Then the following statements are true:
\begin{enumerate}
\item $(\exists A \in Form^{Ag})(Proven(t,A) \in end(m \sqcap h))
\Leftrightarrow t \in Act_m$;

\item If $|m| > 0$ and $h, h' \in H_m$, then $Act(m,h) =
Act(m,h')$;

\item If $h,h' \in H_m$ and $m \sqcap h = m \sqcap h'$, then
$Act(m,h) = Act(m,h')$.
\end{enumerate}
\end{lemma}
\begin{proof}
(Part 1). Assume that for some $A \in Form^{Ag}$ we have
$Proven(t,A) \in end(m \sqcap h)$. Then, by Corollary
\ref{p-mod}.1, we must also have $Proven(t,A) \in end_m$ whence,
by definition of $Act$, it follows that $t \in Act_m$. In the
other direction, let $t \in Act_m$. If for some $g \in H_m$ and
some $A \in Form^{Ag}$ we have $Proven(t,A) \in end(m \sqcap g)$,
then, by Corollary \ref{p-mod}.1, we must also have $Proven(t,A)
\in end_m$, whence also $Proven(t,A) \in end(m \sqcap h)$.
Otherwise, we obtain a contradiction.

Indeed, if for no $g \in H_m$ and $A \in Form^{Ag}$ we have
$Proven(t,A) \in end(m \sqcap g)$, then, by definition of $Act$,
for every $g \in H_m$ there must be an $A_g \in Form^{Ag}$ and a
$j_g \in Ag$ such that $Prove(j_g, t, A_g) \in end(m \sqcap g)$.
Now, consider $A_h$. We have $Prove(j_h, t, A_h) \in end(m \sqcap
h)$, therefore, by \eqref{B9} and $\mathcal{CS}$-maxiconsistency
of $end(m \sqcap h)$, we must also have $t\co A_h \in end(m \sqcap
h)$. It follows by  Lemma \ref{p-mod}.1, that $t\co A_h \in
end_m$. Therefore, given the $\mathcal{CS}$-maxiconsistency of end
sets in elements, we must have $Prove(j_g, t, A_g)\wedge t\co A_h
\in end_m$. By \eqref{B10}, we further get that $Prove(j_g, t,
A_h) \in end_m$. Note that $\{j_g \mid g \in H_m \} \subseteq Ag$
and hence must be finite. Therefore $\bigvee_{g \in
H_m}(Prove(j_g,t,A_h))$ is in fact a finite disjunction, and,
again using $\mathcal{CS}$-maxiconsistency of of end sets in
elements, we obtain that $\bigvee_{g \in H_m}Prove(j_g,t,A_h) \in
end_m$, whence, by Lemma \ref{element-new}.1 and Corollary
\ref{c-sqcap}, it follows that $\Box(\bigvee_{g \in
H_m}(Prove(j_g,t,A_h)) \in end(m \sqcap h)$. By \eqref{T5}, the
latter is in contradiction with $\mathcal{CS}$-maxiconsistency of
$end(m \sqcap h)$, which finishes the proof of Part 1.

(Part 2). In the assumptions of this part, we get, by Corollary
\ref{p-mod}.1, that:

\begin{align*}
t \in Act(m,h) &\Leftrightarrow (\exists A \in
Form^{Ag})(Proven(t,A) \in end(m
\sqcap h))\\
&\Leftrightarrow (\exists A \in Form^{Ag})(Proven(t,A) \in
end_m)\\
&\Leftrightarrow (\exists A \in Form^{Ag})(Proven(t,A) \in end(m
\sqcap h'))\\
&\Leftrightarrow t \in Act(m,h'),
\end{align*}

for an arbitrary $t \in Pol$.

(Part 3). Note that $Act(m,h)$ and $Act(m,h')$ are fully
determined by $end( m \sqcap h)$ and $end( m \sqcap h')$,
respectively, and that, by our assumptions, we must have $end( m
\sqcap h) = end( m \sqcap h')$.
\end{proof}

We now check that the remaining semantic constraints on jstit
models are satisfied:

\begin{lemma}\label{act}
$\mathcal{M}^{Ag}_{\mathcal{CS}}$ satisfies the constraints as to
the expransion of presented proofs, no new proofs guaranteed,
presenting a new proof makes histories divide, and epistemic
transparency of presented proofs.
\end{lemma}
\begin{proof}
We consider the \textbf{expansion of presented proofs} first. Let
$m' \lhd m$ and let $h \in H_m$. If $m' = \dag$, then  we have
$Act(\dag, h) = \emptyset$, so that the expansion of presented
proofs holds. If $m' \neq \dag$, then $m$ is also standard.
Consider then $m' \sqcap h$ and $m \sqcap h$. Both these elements
must be in the basic sequence $]h[$, therefore, one of them must
be an initial segment of another. By $m' \lhd m$ we know that $m'
\sqcap h$ must be a proper initial segment of $m \sqcap h$. So we
may assume that $m' \sqcap h = (\Gamma_1,\ldots, \Gamma_k)$ and $m
\sqcap h = (\Gamma_1,\ldots, \Gamma_n)$ for some appropriate
$\Gamma_1,\ldots, \Gamma_n \subseteq Form^{Ag}$ and $n > k$. Now,
if $t \in Act(m',h)$, then for some $A \in Form^{Ag}$ we must have
either that $Prove(j,t,A) \in end(m' \sqcap h) = \Gamma_k$ for
some $j \in Ag$, or that $Proven(t,A) \in end(m' \sqcap h)$. In
the latter case we will also have $KProven(t,A) \in end(m' \sqcap
h)$ by \eqref{B11} and $\mathcal{CS}$-maxiconsistency of
$\Gamma_k$. Then, since $(\Gamma_1,\ldots, \Gamma_n)$ is an
element, we must have $KProven(t,A) \in \Gamma_n$, whence, by
\eqref{A7} and $\mathcal{CS}$-maxiconsistency of $\Gamma_n$, we
further obtain that $Proven(t,A) \in \Gamma_n = end(m \sqcap h)$.
Hence we must have $t \in Act(m,h)$.

In the former case we also invoke the fact that $(\Gamma_1,\ldots,
\Gamma_n)$ is an element, which in this case directly entails that
$Proven(t,A) \in \Gamma_{k+1}$ and, given that $k + 1 \leq n$, the
rest is the same as in the previous case.

We consider next the \textbf{no new proofs guaranteed} constraint.
Let $m \in Tree$. If $m = \dag$, then $Act_m = \bigcup_{m' \lhd m,
h \in H_m}(Act(m',h)) = \emptyset$ and the constraint is trivially
satisfied. Assume that $m \neq \dag$ and that $t \in Act_m$ and
choose an arbitrary $h \in H_m$. Consider $m \sqcap h$, say $m
\sqcap h =(\Gamma_1,\ldots, \Gamma_n)$. We get then that $m =
([(\Gamma_1,\ldots, \Gamma_n)]_\equiv, k)$ for some $k \in
\omega$. By Lemma \ref{act-technical}.1, we further obtain that
for some $A \in Form^{Ag}$ we will have $Proven(t,A) \in
\Gamma_n$. Now, consider $m' = ([(\Gamma_1,\ldots,
\Gamma_n)]_\equiv, k+1)$. We clearly have $m' \lhd m$, therefore,
by Lemma \ref{technical}.3 $h \in H_{m'}$. Moreover, it is clear
that $m' \sqcap h$ also equals $(\Gamma_1,\ldots, \Gamma_n)$, so
that we get $Proven(t,A) \in \Gamma_n = end(m' \sqcap h)$, whence
$t \in Act(m',h)$ as desired.

We turn next to the \textbf{presenting a new proof makes histories
divide} constraint. Consider $m, m' \in Tree$ such that $m \lhd
m'$ and arbitrary $h, h' \in H_{m'}$. We immediately get then that
$h, h' \in H_{m}$.  If $m = \dag$, then the constraint is verified
trivially. If $m \neq \dag$, then we have two cases to consider:

\emph{Case 1}. $\overrightarrow{m} = \overrightarrow{m'}$ and $|m|
> |m'|$. Then we must have $|m| > 0$, and by Lemma \ref{act-technical}.2 it follows
that in this case for all $h, h' \in H_m$ we will have $Act(m, h)
= Act(m,h')$ so that the constraint is verified.

\emph{Case 2}. There is a $\xi \in \overrightarrow{m}$ such that
$\xi$ is a proper initial segment of every $\tau \in
\overrightarrow{m'}$. Consider then $m' \sqcap h$ and $m' \sqcap
h'$. These are elements in $\overrightarrow{m'}$, and hence $\xi$
is a proper initial segment of both $m' \sqcap h$ and $m' \sqcap
h'$. Moreover, we know that $m \sqcap h$ is also in $]h[$ and
hence must be an initial segment of $m' \sqcap h$  of the length
equal to the length of $\xi$. The same holds for $m \sqcap h'$ and
$m' \sqcap h'$, respectively. It follows that $m \sqcap h = m
\sqcap h' = \xi$ whence, by Lemma \ref{act-technical}.3, we
immediately get $Act(m,h) = Act(m,h')$.

It remains to check the \textbf{epistemic transparency of
presented proofs} constraint. Assume that $m, m' \in Tree$ are
such that $R(m,m')$. If we have $m = \dag$, then, by definition,
we must have $Act_m = \emptyset$, and the constraint is verified
in a trivial way. If, on the other hand, $m \neq \dag$, then, by
$R(m,m')$, we must also have $m' \neq \dag$. Assume that $t \in
Act_m$. Then, by Lemma \ref{act-technical}.1, we also have
$Proven(t,A) \in end(m \sqcap h)$ for some $A \in Form^{Ag}$. By
$R(m,m')$ and Corollary \ref{c-sqcap}, it follows that
$Proven(t,A) \in \bigcap_{g \in H_{m'}}end(m \sqcap g)$, whence,
again by Lemma \ref{act-technical}.1, we know that $t \in
Act_{m'}$.
\end{proof}

\subsection{The truth lemma}
It follows from Lemmas \ref{leq}, \ref{choice}, \ref{e}, and
\ref{act}, that our above-defined canonical model
$\mathcal{M}^{Ag}_{\mathcal{CS}}$ is a unirelational jstit model
for $Ag$. By Lemma \ref{proven} we know that
$\mathcal{M}^{Ag}_{\mathcal{CS}}$ is $\mathcal{CS}$-normal. Now we
need to supply a truth lemma:
\begin{lemma}\label{truth}
Let $A \in Form$, let $m \in Tree \setminus \{ \dag \}$ be such
that $|m| = 0$, and let $h \in H_m$. Then:
$$
\mathcal{M}^{Ag}_{\mathcal{CS}},m,h \models A \Leftrightarrow A
\in end(m \sqcap h).
$$
\end{lemma}
\begin{proof}
As is usual, we prove the lemma by induction on the construction
of $A$. The basis of induction with $A = p \in Var$ we have by
definition of $V$, whereas Boolean cases for the induction step
are trivial. The cases for $A$ being of the form $\Box B$, $[j]B$,
$KB$ or $t\co B$ for some $j \in Ag$ and $t \in Pol$, are treated
exactly as for JA-STIT (see \cite{OLexpl-gen}, although Corollary
\ref{p-mod} already provides a clear hint for the proof of cases
$A = \Box B, [j]B$). We consider the cases for the two proving
modalities:

\emph{Case 1}. $A = Proven(t,B)$ for some $t \in Pol$. If
$Proven(t,B) \in end(m \sqcap h)$, then, by Lemma
\ref{act-technical}.1, we must have $t \in Act_m$. Furthermore, by
$\mathcal{CS}$-maxiconsistency of $end(m \sqcap h)$ and
\eqref{B11}, we must also have $t\co B \in end(m \sqcap h)$,
whence by the induction case for $t \co B$ it follows that
$\mathcal{M}^{Ag}_{\mathcal{CS}},m,h \models t\co B$. Summing this
up with $t \in Act_m$, we get that
$\mathcal{M}^{Ag}_{\mathcal{CS}},m,h \models Proven(t,B)$.

In the other direction, assume that $Proven(t,B) \notin end(m
\sqcap h)$. We have then two subcases to consider:

\emph{Case 1.1}. $t\co B \notin end(m \sqcap h)$. Then, by the
induction case for $t \co B$ it follows that
$\mathcal{M}^{Ag}_{\mathcal{CS}},m,h \not\models t\co B$, whence,
by definition of satisfaction relation, we get that
$\mathcal{M}^{Ag}_{\mathcal{CS}},m,h \not\models Proven(t,B)$.

\emph{Case 1.2}. $t\co B \in end(m \sqcap h)$. Now, if $t \in
Act_m$, then, by Lemma \ref{act-technical}.1, there must be some
$C \in Form^{Ag}$ such that $Proven(t, C) \in end(m \sqcap h)$. By
\eqref{B12} and $\mathcal{CS}$-maxiconsistency of $end(m \sqcap
h)$, this means that $Proven(t,B) \in end(m \sqcap h)$ contrary to
our assumption. The obtained contradiction shows that $t \notin
Act_m$, whence $\mathcal{M}^{Ag}_{\mathcal{CS}},m,h \not\models
Proven(t,B)$, as desired.

\emph{Case 2}. $A = Prove(j,t,B)$ for some $j \in Ag$ and $t \in
Pol$. If $Prove(j,t,B) \in end(m \sqcap h)$, then, by
$\mathcal{CS}$-maxiconsistency of $end(m \sqcap h)$ and
\eqref{B9}, we must have
\begin{equation}\label{E:fin0}
t \co B, \neg\Box Prove(j,t,B), \neg Proven(t,B) \in end(m \sqcap
h).
\end{equation}
This immediately gives us, by the induction case for $t \co B$,
that:
\begin{equation}\label{E:fin1}
    \mathcal{M}^{Ag}_{\mathcal{CS}},m,h \models t\co B.
\end{equation}
Moreover, we can infer by Corollary \ref{p-mod}.2 that
$Prove(j,t,B) \in end_{(m,h,j)}$, whence it follows, by $|m| = 0$
and the definition of $Act$, that:
\begin{equation}\label{E:fin2}
t \in Act_{(m,h,j)}.
\end{equation}
Next, we observe that $\neg\Box Prove(j,t,B) \in end(m \sqcap h)$
yields, by $\mathcal{CS}$-maxiconsistency of $end(m \sqcap h)$,
that $\Box Prove(j,t,B) \notin end(m \sqcap h)$, whence, by Lemma
\ref{element-new}.1 and Corollary \ref{c-sqcap} we get that
$Prove(j,t,B) \notin end_m$. Therefore, we choose a $g \in H_m$
such that $Prove(j,t,B) \notin end(m \sqcap g)$. By
$\mathcal{CS}$-maxiconsistency of $end(m \sqcap g)$ and
\eqref{B13} we get that:
\begin{equation}\label{E:fin3}
    \langle j\rangle\bigwedge_{i \in Ag}\neg Prove(i,t,B) \in end(m
\sqcap g).
\end{equation}
It follows by Lemma \ref{element-new-stit}, that we can choose a
$g' \in Choice^m_j(g) \subseteq H_m$ such that:
\begin{equation}\label{E:fin4}
    \bigwedge_{i \in Ag}\neg Prove(i,t,B) \in end(m
\sqcap g').
\end{equation}
We now show that assuming $t\in Act(m,g')$ leads to a
contradiction. Indeed, if $t\in Act(m,g')$, then we would have
either $Prove(i,t,C)$ or $Proven(t,C)$ in $end(m \sqcap g')$ for
some $i \in Ag$ and $C \in Form^{Ag}$. In the former case note
that, by Corollary \ref{p-mod}.1, it would follow from
\eqref{E:fin0} that $t\co B \in end(m \sqcap g')$, whence by
$\mathcal{CS}$-maxiconsistency of $end(m \sqcap g')$ and
\eqref{B10} we would further obtain that $Prove(i,t,B) \in end(m
\sqcap g')$, contrary to \eqref{E:fin4}. In the latter case we
would have $Proven(t, C)\in end(m \sqcap h)$, by $g', h \in H_m$,
and Corollary \ref{p-mod}.1. In virtue of
$\mathcal{CS}$-maxiconsistency of $end(m \sqcap h)$,
\eqref{E:fin1}, and \eqref{B12} this would further imply that
$Proven(t, B)\in end(m \sqcap h)$, in contradiction with
\eqref{E:fin0}. The obtained contradiction shows that we must
have:
\begin{equation}\label{E:fin5}
 t\notin Act(m,g').
\end{equation}
Adding up \eqref{E:fin1}, \eqref{E:fin2}, and \eqref{E:fin5},
yields that   $\mathcal{M}^{Ag}_{\mathcal{CS}},m,h \models
Prove(j,t,B)$.

In the other direction, assume that $Prove(j,t,B) \notin end(m
\sqcap h)$. We have to consider three further subcases.

\emph{Case 2.1}. $t\co B \notin end(m \sqcap h)$. Then, by
induction case for $t\co B$, it follows that
$\mathcal{M}^{Ag}_{\mathcal{CS}},m,h \not\models t\co B$, whence
by \eqref{B9} we get that $\mathcal{M}^{Ag}_{\mathcal{CS}},m,h
\not\models Prove(j,t,B)$.

\emph{Case 2.2}. $t\co B, Proven(t,B) \in end(m \sqcap h)$. Then
it follows, by Case 1, that $\mathcal{M}^{Ag}_{\mathcal{CS}},m,h
\models Proven(t,B)$, whence, again by \eqref{B9}, we get that
$\mathcal{M}^{Ag}_{\mathcal{CS}},m,h \not\models Prove(j,t,B)$.

\emph{Case 2.3}. $t\co B, \neg Proven(t,B)  \in end(m \sqcap h)$.
Then, by Corollary \ref{p-mod}.1 and
$\mathcal{CS}$-maxiconsistency of $end(m \sqcap h)$, it follows
that $t\co B \in end_m$. Now, observe that $Prove(j,t,B) \notin
end(m \sqcap h)$ implies, by $\mathcal{CS}$-maxiconsistency of
$end(m \sqcap h)$ and \eqref{B13}, that:
\begin{equation}\label{E:fin6}
    \langle j\rangle\bigwedge_{i \in Ag}\neg Prove(i,t,B) \in end(m
\sqcap h).
\end{equation}
It follows by Lemma \ref{element-new-stit}, that we can choose a
$g \in Choice^m_j(h)$ such that:
\begin{equation}\label{E:fin7}
    \bigwedge_{i \in Ag}\neg Prove(i,t,B) \in end(m
\sqcap g).
\end{equation}
We now show that assuming $t\in Act(m,g)$ leads to a
contradiction. Indeed, if $t\in Act(m,g)$, then we would have
either $Prove(i,t,C)$ or $Proven(t,C)$ in $end(m \sqcap g)$ for
some $i \in Ag$ and $C \in Form^{Ag}$. In the former case note
that we also have $t\co B \in end(m \sqcap g)$, whence by
$\mathcal{CS}$-maxiconsistency of $end(m \sqcap g)$ and
\eqref{B10} we would further obtain that $Prove(i,t,B) \in end(m
\sqcap g')$, contrary to \eqref{E:fin7}. In the latter case, again
using $t\co B \in end(m \sqcap g)$, we would have $Proven(t, B)\in
end(m \sqcap g)$, by $\mathcal{CS}$-maxiconsistency of $end(m
\sqcap g)$ and \eqref{B12}, which, by Corollary \ref{p-mod}.1 and
the fact that $g \in H_m$ would mean that $Proven(t, B)\in end(m
\sqcap h)$, making $end(m \sqcap h)$ $\mathcal{CS}$-inconsistent.
The obtained contradiction shows that we must have:
\begin{equation}\label{E:fin9}
 t\notin Act(m,g),
\end{equation}
and since $g \in Choice^m_j(h)$, this immediately implies that
$\mathcal{M}^{Ag}_{\mathcal{CS}},m,h \not\models Prove(j,t,B)$, as
desired.

This finishes the list of the modal induction cases at hand, and
thus the proof of our truth lemma is complete.
\end{proof}

\section{The main result and conclusions}\label{main}

We are now in a position to prove Theorem \ref{completeness}. The
proof proceeds as follows. One direction of the theorem was proved
as Corollary \ref{c-soundness}. In the other direction, assume
that $\Gamma \subseteq Form^{Ag}$ is $\mathcal{CS}$-consistent.
Then $\Gamma$ can be extended to a $\mathcal{CS}$-maxiconsistent
$\Delta$. But then consider $\mathcal{M}^{Ag}_{\mathcal{CS}} =
\langle Tree, \unlhd, Choice, Act, R, \mathcal{E}, V\rangle$, the
canonical model defined in Section \ref{canonicalmodel}. It is
clear that $(\Delta)$ is an element, therefore $m =
([(\Delta)]_\equiv, 0) \in Tree$. By Lemma \ref{sqcap}, there is a
history $h \in H_m$ such that $(\Delta) = ([(\Delta)]_\equiv, 0)
\sqcap h$. For this $h$, we will also have $\Delta =
end(([(\Delta)]_\equiv, 0) \sqcap h)$. By Lemma \ref{truth}, we
therefore get that:
$$
\mathcal{M}, ([(\Delta)]_\equiv, 0), h \models \Delta \supseteq
\Gamma,
$$
and thus $\Gamma$ is shown to be satisfiable in a
$\mathcal{CS}$-normal jstit unirelational model, hence in a normal
jstit model.

\textbf{Remark}. Note that the canonical model used in this proof
is universal in the sense that it satisfies every subset of
$Form^{Ag}$ which is $\mathcal{CS}$-consistent.

As an obvious corollary of Theorem \ref{completeness} we get the
compactness property.

Thus, building up on an earlier work on justification stit
formalisms, we have defined the explicit fragment of basic jstit
logic introduced in \cite{OLWA}. For this logic we have presented
a strongly complete axiomatization which is stable relative to
extensions with constant specifications. This result is similar to
the completeness theorem obtained earlier for JA-STIT in
\cite{OLexpl-gen} and also borrows from this paper some techniques
and results related to the construction of canonical model. We
also note that Proposition \ref{proposition} proven above in
Section \ref{basic} shows that explicit jstit logic, just like
JA-STIT, can distinguish between the class of models with a
discrete temporal substructure and the general class of models,
even though it apparently has less expressive powers. We observe
that the formula $A$ used in the proof of Proposition
\ref{proposition} is clearly related to the formula $K(\neg \Box
Ex \vee \Box Ey) \to (\neg Ex \vee Ey)$ used to prove a similar
proposition for JA-STIT in \cite{OLexpl-gen}. The latter formula
was shown to admit of an easy generalization which led to an
axiomatization of JA-STIT over the class of jstit models based on
discrete time. It would be natural to look for an axiomatization
of explicit jstit logic over the same class of models. However,
this time the reduced expressive power of explicit jstit logic may
actually prove to be an obstacle, since the generalization pattern
which we applied in the case of JA-STIT does not seem to work in
the case of explicit jstit logic.

Another natural, but even more uphill, task for the future
research would be to try and attempt an axiomatization of full
basic jstit logic now that we have positive experience with
axiomatizations of both its implicit and its explicit fragment.

\section{Acknowledgements}
To be inserted.

}

\end{document}